\documentclass[12pt,reqno]{amsart}
\usepackage[T1]{fontenc}
\usepackage{amsfonts}
\usepackage{amssymb}
\usepackage{mathrsfs}
\usepackage[latin1]{inputenc}
\usepackage{amsmath}
\usepackage[english]{babel}
\usepackage{setspace}

\newtheorem{theorem}{Theorem}[section]
\newtheorem{lemma}[theorem]{Lemma}

\newtheorem{remark}[theorem]{Remark}
\newtheorem{prop}[theorem]{Proposition}
\newtheorem{example}[theorem]{Example}
\newtheorem{corollary}[theorem]{Corollary}
\newtheorem{examples}[theorem]{Examples}

\numberwithin{equation}{section}

\newcommand{\R}{{\mathbb R}}

\newcommand{\C}{{\mathbb C}}
\newcommand{\N}{{\mathbb N}}

\newcommand{\cL}{{\mathcal L}}

\newcommand{\cK}{{\mathcal K}}

\newcommand{\cU}{{\mathcal U}}

\newcommand{\ve}{\varepsilon}

\newcommand{\su}{\subseteq}

\newcommand{\sC}{\mathsf{C}}
\newcommand{\ov}{\overline}

\newcommand\Ker{\mathop{\rm Ker}}

\begin{document}

\title{The  Ces\`aro operator  in weighted $\ell_1$  spaces}

\author{Angela\,A. Albanese, Jos\'e Bonet, Werner \,J. Ricker}

\thanks{This article is accepted for publication in Mathematische Nachrichten DOI: 10.1002/mana.201600509}

\address{ Angela A. Albanese\\
Dipartimento di Matematica e Fisica ``E. De Giorgi''\\
Universit\`a del Salento- C.P.193\\
I-73100 Lecce, Italy}
\email{angela.albanese@unisalento.it}

\address{Jos\'e Bonet \\
Instituto Universitario de Matem\'{a}tica Pura y Aplicada
IUMPA \\
Universitat Polit\`ecnica de
Val\`encia\\
E-46071 Valencia, Spain} \email{jbonet@mat.upv.es}

\address{Werner J.  Ricker \\
Math.-Geogr. Fakult\"{a}t \\
 Katholische Universit\"{a}t
Eichst\"att-Ingol\-stadt \\
D-85072 Eichst\"att, Germany}
\email{werner.ricker@ku-eichstaett.de}
\markboth{A.\,A. Albanese, J. Bonet, W. \,J. Ricker}%
{\MakeUppercase{ }}

\begin{abstract}
Unlike for $\ell_p$, $1<p\leq\infty$, the discrete Ces\`aro operator $\sC$ does \textit{not} map $\ell_1$ into itself. We identify precisely those weights $w$ such that $\sC$ does map $\ell_1(w)$ continuously into itself. For these weights a complete description of the eigenvalues and the spectrum of $\sC$ are presented. It is also possible to identify all $w$ such that $\sC$ is a compact operator in $\ell_1(w)$. The final section  investigates  the mean ergodic properties of $\sC$ in $\ell_1(w)$. Many examples are presented in order to supplement the results and to illustrate the  phenomena that occur.
\end{abstract}

\maketitle

\markboth{A.\,A. Albanese, J. Bonet and W.\,J. Ricker}%
{\MakeUppercase{The discrete Ces\`aro operator in weighted $\ell_1$  spaces }}

\section{Introduction}

The discrete Ces\`aro operator $\sC$ is defined on the linear space $\C^\N$ (consisting of all scalar sequences) by
\begin{equation}\label{eq.Ce-op}
\sC x:=\left(x_1, \frac{x_1+x_2}{2}, \ldots, \frac{x_1+\ldots +x_n}{n}, \ldots\right),\quad x=(x_n)_{n\in\N}\in \C^\N.
\end{equation}
The operator $\sC$ is said to \textit{act} in a vector subspace  $X\su\C^\N$ if it maps $X$ into itself. Of particular interest is the situation when $X$ is a Banach space. Two fundamental questions in this case are: Is $\sC\colon X\to X$ continuous and, if so, what is its spectrum? Amongst the classical Banach spaces $X\su \C^\N$ where answers are known we mention $\ell_p$ ($1<p<\infty$), \cite{BHS}, \cite{L}, and $c_0$, \cite{L}, \cite{R},  both $c$, $\ell_\infty$, \cite{A-B}, \cite{L}, as well as $ces_p$, $p\in \{0\}\cup (1,\infty)$, \cite{C-R},  the spaces of bounded variation $bv_0$, \cite{Ok}, and $bv_p$, $1\leq p<\infty$, \cite{A-B2}, and the Bachelis spaces $N^p$, $1<p<\infty$, \cite{CR-1}. For $\sC$ acting in the \textit{weighted} Banach spaces $\ell_p(w)$, $1<p<\infty$, and $c_0(w)$ we refer to \cite{ABR}, \cite{ABR-9}. There is no claim that this list of spaces (and references) is complete; see also \cite{B}.

One  aim of this paper is to investigate the two questions mentioned above for $\sC$ acting in the weighted Banach space $\ell_1(w)$. Unlike for the setting of $\ell_p(w)$, $1<p<\infty$,  where the corresponding paradigm space for $\sC$ is $\ell_p$, $1<p<\infty$,  the ``paradigm space'' $\ell_1$ is \textit{not available} as a guideline for $\sC$ in $\ell_1(w)$ because $\sC$ does \textit{not} act in $\ell_1$. Hence, it is unclear what to expect when $\sC$ acts in $\ell_1(w)$.

So, let $w=(w(n))_{n=1}^\infty$ be a  sequence, always assumed to be \textit{bounded} and \textit{strictly} positive. Define the vector space
\[
\ell_1(w):=\left\{x=(x_n)_{n\in \N}\in \C^\N\colon \sum_{n=1}^\infty w(n)|x_n|<\infty\right\},
\]
equipped with the norm  $\|x\|_{1,w}:=\sum_{n\in\N}w(n)|x_n|$, for   $x\in \ell_1(w)$.
Then  $\ell_1(w)$  is isometrically isomorphic to  $\ell_1$ via the linear multiplication operator $\Phi_w\colon \ell_1(w)\to \ell_1$ given by
\[
 x=(x_n)_{n\in\N}\mapsto \Phi_w(x):=(w(n)x_n)_{n\in\N}.
\]
Accordingly,  $\ell_1(w)$ is a  weakly sequentially complete Banach space with the Schur property, \cite[pp.218--220]{ME}. Its  dual  space $(\ell_1(w))'$   is the Banach space $\ell_{\infty}(u)$ with the norm
$\|x\|_{\infty,u}:=\sup_{n\in\N}u(n)|x_n|$, for   $x=(x_n)_{n\in \N}\in \ell_\infty(u)$,
where  $u(n):=w(n)^{-1}$ for $n\in\N$. The closed subspace
\[
\left\{x=(x_n)_{n\in \N}\in \C^\N\colon \lim_{n\to\infty} u(n)|x_n|=0\right\}
\]
of $\ell_{\infty}(u)$ is denoted by $c_{0}(u)$ and the restriction of the norm $\|\cdot \|_{\infty,u}$ to $c_0(u)$ is written as $\|\cdot \|_{0,u}$. Of course, the bidual $c_0(u)''=\ell_\infty(u)$ and the dual $c_0(u)'=\ell_1(w)$. Clearly, the Banach spaces $\ell_\infty(u)$ and $c_0(u)$ are also defined and have the above mentioned properties for \textit{every} strictly positive sequence $u=(u(n))_{n\in\N}$, not just for $u=w^{-1}$.
 The canonical vectors $e_k:=(\delta_{kn})_{n\in\N}$, for $k\in\N$, form an unconditional basis in  $\ell_1(w)$. Consequently, whenever $\sC$ does act in $\ell_1(w)$, then it is necessarily continuous  (via the Closed Graph Theorem).
If $\inf_{n\in\N}w(n)>0$, then $\ell_1(w)=\ell_1$ with equivalent norms and so we are in a space in which $\sC$ does \textit{not} act. Accordingly, we are only interested in  the case when $\inf_{n\in\N}w(n)=0$. Of course, $\Phi_w$ is also defined on all of $\C^\N$ in which case it is a vector space isomorphism of $\C^\N$ onto itself.

For any Banach space $X$, let $I$ denote the identity operator on $X$ and $\cL(X)$  the vector space of all continuous linear operators from $X$ into itself. The \textit{spectrum} and the \textit{resolvent set}  of $T\in \cL(X)$ are denoted by $\sigma(T)$ and $\rho(T)$, respectively, \cite[Ch. VII]{DSI}. The set of all \textit{eigenvalues} of $T$, also called the \textit{point spectrum} of $T$, is denoted by $\sigma_{pt}(T)$. The \textit{spectral radius} $r(T):=\sup\{|\lambda|\colon \lambda\in\sigma(T)\}$  always satisfies $r(T)\leq \|T\|$, \cite[p.567]{DSI}.  The ideal of  compact operators from $X$ into a Banach space $Y$ is denoted by $\cK(X,Y)$. If $X=Y$, we simply write $\cK(X)$. The dual Banach space of $X$ is denoted by $X'$
 and the dual operator of $T\in \cL(X)$ by $T'\in \cL(X')$.

In Section 2 we identify all weights $w$ such that $\sC$ acts  in $\ell_1(w)$; see Proposition \ref{P1-C}(i). Necessarily $\inf_{n\in\N}w(n)=0$ (cf. Remark \ref{R.VW}(i)) but this condition is far from sufficient; see Examples \ref{esempi-C}(i), (iii). Moreover, this necessary condition cannot be replaced by $w\in c_0$; see Remark \ref{R.VW}(ii).
 Even if $w$ is a rapidly decreasing sequence it still need \textit{not} follow that $\sC$ acts in $\ell_1(w)$; see Remark \ref{R.W_N}(ii). The \textit{compactness} of $\sC$ in $\ell_1(w)$ is characterized in Proposition \ref{P1-C}(ii). A useful sufficient condition for $w$, ensuring the compactness of $\sC$ in $\ell_1(w)$, is the requirement that
\begin{equation}\label{eq.I-limite}
\limsup_{n\to\infty}\frac{w(n+1)}{w(n)}\in [0,1);
\end{equation}
see Proposition \ref{P2-C}. Applications of \eqref{eq.I-limite}
to particular weights are given in parts (i)-(iv) of Examples \ref{E.COP}. On the other hand, the weights given in (v), (vi) of Examples \ref{E.COP} show that the condition \eqref{eq.I-limite} is not necessary for the compactness of $\sC$ in $\ell_1(w)$. A comparison type result for compactness (and also for continuity) is presented in Proposition \ref{P.Comparison}. The usefulness of this criterion is illustrated via Example \ref{E.ConNOComp}. Somewhat surprisingly, there exist rapidly decreasing weights $w$ for which $\sC$ acts  in $\ell_1(w)$ but, \textit{fails} to be compact; see the weight $v$ in Example \ref{E.Comparison}(ii).

Section 3 investigates the \textit{spectrum} of $\sC$, provided that $\sC$ acts  in $\ell_1(w)$; for brevity we indicate this by writing $\sC^{(1,w)}$ for $\sC$ or $\sC^{(1,w)}\in \cL(\ell_1(w))$. Relevant for determining   $\sigma(\sC^{(1,w)})$ are the sets
\[
R_w:=\{t\in\R\colon \sum_{n=1}^\infty n^tw(n)<\infty\}\ \mbox{ and }\ S_w(1):=\{s\in \R\colon \sup_{n\in\N}\frac{1}{n^sw(n)}<\infty\}.
\]
Whenever $R_w\not=\R$ (resp. $S_w(1)\not=\emptyset$) we define $t_0:=\sup R_w$ (resp. $s_1:=\inf S_w(1)$). Useful connections between $t_0$, $s_1$, the sets $R_w$, $S_w(1)$ and the condition that $w$ is  rapidly decreasing  are presented in Propositions \ref{P2-Sp} and  \ref{P.S_R}. These propositions  are needed to establish the two main results of the section. Theorem \ref{t:sp-dec-N} characterizes $\sigma(\sC^{(1,w)})$ and identifies the point spectrum $\sigma_{pt}(\sC^{(1,w)})$.  It turns out that $\sigma_{pt}(\sC^{(1,w)})=\emptyset$ precisely when $w\not\in \ell_1$ (cf. Remark \ref{R.3n}(iii)). Whenever $\sC^{(1,w)}\in \cL(\ell_1(w))$ and $S_w(1)\not=\emptyset$, necessarily $s_1>0$ and
\[
\left\{\lambda\in \C\colon \left|\lambda-\frac{1}{2s_1}\right|\leq \frac{1}{2s_1}\right\}\cup\left\{\frac{1}{n}\colon n\in\N\right\}\su \sigma(\sC^{(1,w)});
\]
see Proposition \ref{t:sp-dec}. In particular, $\sC^{(1,w)}$ cannot then be compact. This includes such weights as $w_\alpha:=(\frac{1}{n^\alpha})_{n\in\N}$ for all $\alpha>0$ (cf. Example \ref{es.enne}) and others. On the other hand, if $\sC^{(1,w)}$ is compact, then
\begin{equation}\label{e.I-C}
\sigma_{pt}(\sC^{(1,w)})=\left\{\frac{1}{n}\colon n\in\N\right\}\ \mbox{ and }\ \sigma(\sC^{(1,w)})=\{0\}\cup \sigma_{pt}(\sC^{(1,w)})
\end{equation}
and the weight $w$ is necessarily rapidly decreasing.  The converse is not valid in general, i.e., there exist rapidly decreasing weights $w$ such that $\sC^{(1,w)}\in \cL(\ell_1(w))$, the spectra of $\sC^{(1,w)}$ are given by \eqref{e.I-C} but, $\sC^{(1,w)}$ is not compact; see Example \ref{ex.nocompact}. For each $k\in\N$ there exists a weight $w$ such that $\sigma_{pt}(\sC^{(1,w)})=\{\frac{1}{n}\colon 1\leq n\leq k\}$; see Example \ref{es.enne}.

The final section treats various \textit{mean ergodic} properties of $\sC^{(1,w)}$. Also relevant is the \textit{power boundedness} of $\sC^{(1,w)}$, i.e., $\sup_{n\in\N}\|(\sC^{(1,w)})^n\|<\infty$, and the weaker condition of  \textit{Ces\`aro boundedness}  (cf. Section 4 for the definition). We record a few sample results. For instance, if $\sC^{(1,w)}\in \cL(\ell_1(w))$, then $\sC^{(1,w)}$ is power bounded if and only if the sequence of its iterates $\{(\sC^{(1,w)})^n\}_{n\in\N}$ is convergent in the strong operator topology of $\cL(\ell_1(w))$ to the projection onto the null space $\Ker (I- \sC^{(1,w)})$; see Theorem \ref{T-conv}(i).  Moreover, the power boundedness of $\sC^{(1,w)}$implies that $w\in\ell_1$ (cf. Lemma \ref{L-Mean1}).
It is also established that $\sC^{(1,w)}$ is mean ergodic if and only if it is Ces\`aro bounded (cf. Theorem \ref{T-conv}(ii)). Such results do not hold for general Banach space operators; see Remark \ref{R.NSW}. Intimately related to the \textit{uniform} mean ergodicity of  $\sC^{(1,w)}$ (indeed, for any Banach space operator) is the closedness of the range of $I-\sC^{(1,w)}$ in $\ell_1(w)$. Under the natural restriction that $w\in \ell_1$, this property is equivalent to the requirement
\begin{equation}\label{e.I-UME}
\sup_{m\in\N}\frac{1}{mw(m+1)}\sum_{n=m+1}^\infty w(n)<\infty;
\end{equation}
see Proposition \ref{L.range_closed}.
The condition \eqref{e.I-UME} also suffices for $\sC^{(1,w)}$ to be  both power bounded and uniformly mean ergodic (cf. Proposition \ref{P.SC-Bound}). According to Proposition \ref{P.SC-Bound-C}, the compactness of $\sC^{(1,w)}$ always implies that \eqref{e.I-UME} is satisfied; the converse is not true in general (cf. Example \ref{R.L-Compact}).

An effort has been made to present many and varied examples, both to supplement the  results and to illustrate the phenomena that occur.

\section{Continuity and compactness of $\sC$ in $\ell_1(w)$}

Given two  strictly positive sequences $v=(v(n))_{n=1}^\infty$  and $w=(w(n))_{n=1}^\infty$, let $T_{v,w}\colon \C^\N\to \C^\N$ denote the linear operator defined by
\begin{equation}\label{e.opau}
T_{v,w}x:=\left(\frac{w(n)}{n}\sum_{k=1}^n\frac{x_k}{v(k)}\right)_{n\in\N},\quad x=(x_n)_{n\in\N}\in \C^\N.
\end{equation}
Observe that $\Phi_w	\sC=T_{v,w}\Phi_v$ as linear maps on $\C^\N$. Hence, the Ces\`aro operator $\sC=\Phi_w^{-1}T_{v,w}\Phi_v$  maps $\ell_1(v)$ continuously (resp., compactly) into $\ell_1(w)$ if and only if the restricted operator $T_{v,w}\in \cL(\ell_1)$ (resp., $T_{v,w}\in \cK(\ell_1)$). In this regard the following  result will be useful,
\cite[p.11]{KL}, \cite[Lemma 2]{R}, \cite[p.220]{T}.

\begin{lemma}\label{L_U} Let $A=(a_{mn})_{m,n\in\N}$ be a matrix with entries from $\C$ and $T\colon \C^\N\to\C^\N$ be the linear operator defined by
\[
Tx:=\left(\sum_{n=1}^\infty a_{mn}x_n\right)_{m\in\N},\quad x=(x_n)_{n\in\N},
\]
interpreted as $Tx\in \C^\N$  exists for  $x\in \C^\N$. Then $T\in \cL(\ell_1)$ if and only if

\[
\sup_{n\in\N}\sum_{m=1}^\infty |a_{mn}|<\infty.
\]
In this case, the operator  norm of $T$ is given by $\|T\|=\sup_{n\in\N}\sum_{m=1}^\infty |a_{mn}|$.
\end{lemma}

An immediate application is the following  result.

\begin{prop}\label{P1-C} Let $v=(v(n))_{n=1}^\infty$  and $w=(w(n))_{n=1}^\infty$ be two bounded, strictly positive sequences.
\begin{itemize}
\item[\rm (i)]  $\sC$  maps $\ell_1(v)$ continuously into $\ell_1(w)$ if and only
 if
\begin{equation}\label{e.Co-C-1}
M_{v,w}:=\sup_{n\in\N}\frac{1}{v(n)}\sum_{m=n}^\infty \frac{w(m)}{m}<\infty.
\end{equation}
In this case, $\|\sC\|=M_{v,w}$.
\item[\rm (ii)]  $\sC$  maps $\ell_1(v)$ compactly into $\ell_1(w)$ if and only if
\begin{equation}\label{e.Comp}
\lim_{n\to\infty}\frac{1}{v(n)}\sum_{m=n}^\infty \frac{w(m)}{m}=0.
\end{equation}
\end{itemize}
\end{prop}

\begin{proof} (i) By the remark prior to Lemma \ref{L_U} we only need to show  that the operator $T_{v,w}\in \cL(\ell_1)$ if and only if \eqref{e.Co-C-1} is satisfied.

Now, $T_{v,w}=\Phi_w\sC\Phi_v^{-1}$ is defined via the matrix $A:=\left(a_{mn}\right)_{m,n\in\N}$ where, for each $m\in\N$, $a_{mn}:=\frac{w(m)}{mv(n)}$ for $1\leq n\leq m$ and $a_{mn}:=0$ otherwise.
 According to Lemma \ref{L_U},  $T_{v,w}\in \cL(\ell_1)$ if and only if
\[
\sup_{n\in\N}\frac{1}{v(n)}\sum_{m=n}^\infty \frac{w(m)}{m}<\infty,
\]
i.e., if and only if  \eqref{e.Co-C-1} is satisfied, in which case $\|T_{v,w}\|=M_{v,w}<\infty$.

So, assume now that $M_{v,w}<\infty$, in which case $\|T_{v,w}\|=M_{v,w}$. Then the identity $\sC=\Phi_w^{-1}T_{v,w}\Phi_v$ together with the fact that both $\Phi_v$ and $\Phi_w^{-1}$ are isometric isomorphisms, implies that $\|\sC\|=M_{v,w}$.

(ii)
Assume first that $\sC\in \cK(\ell_1(v),\ell_1(w))$. In particular,
$\sC$ is also continuous and so  \eqref{e.Co-C-1} is satisfied with $M_{v,w}<\infty$. The claim is that the operator $A\colon c_0(w^{-1})\to c_0(v^{-1})$ defined by $Ay:=\left(\sum_{m=n}^\infty\frac{y_m}{m}\right)_{n\in\N}$, for $y\in c_0(w^{-1})$,  is then continuous and its dual operator $A'$ is precisely $\sC\colon \ell_1(v)\to \ell_1(w)$. To establish continuity, fix $y\in c_0(w^{-1})$. Let $\ve>0$. Select $n_0\in\N$ such that $|y_n|w(n)^{-1}<\ve/M_{v,w}$ for all $n\geq n_0$. It follows, for every $n\geq n_0$, that
\[
\frac{1}{v(n)}\left|\sum_{m=n}^\infty\frac{y_m}{m}\right|\leq \frac{1}{v(n)}\sum_{m=n}^\infty\frac{|y_m|}{w(m)}\frac{w(m)}{m}<\frac{\ve}{M_{v,w}}\frac{1}{v(n)}\sum_{m=n}^\infty\frac{w(m)}{m}\leq \ve.
\]
Accordingly, $Ay\in c_0(v^{-1})$. Moreover, for each  $n\in\N$ we have  that
\[
\frac{1}{v(n)}\left|\sum_{m=n}^\infty\frac{y_m}{m}\right|\leq \frac{1}{v(n)}\sum_{m=n}^\infty\frac{|y_m|}{w(m)}\frac{w(m)}{m}\leq \|y\|_{0, w^{-1}}\frac{1}{v(n)}\sum_{m=n}^\infty\frac{w(m)}{m}\leq M_{v, w}\|y\|_{0, w^{-1}},
\]
which yields
$\|Ay\|_{0, v^{-1}}\leq  M_{v, w}\|y\|_{0, w^{-1}}$.
But, $y\in c_0(w^{-1})$ is arbitrary, and so   $A$ is continuous with $\|A\|\leq M_{v, w}$.
It is routine to check that $A'=\sC$.

Since $\sC\in \cK(\ell_1(v),\ell_1(w))$ and $\sC$ is the dual operator of $A$, Schauder's theorem implies  that $A\in \cK(c_0(w^{-1}),c_0(v^{-1}))$, \cite[Theorem 3.4.15]{ME},
\cite[p.282]{Y}. In particular,  $A\in\cL(c_0(w^{-1}),c_0(v^{-1}))$ is necessarily  weakly compact.
Hence,
its bidual operator $A''=\sC'\in \cL(\ell_\infty(w^{-1}),\ell_\infty(v^{-1}))$ actually maps $\ell_\infty(w^{-1})$ into $c_0(v^{-1})$, \cite[Theorem 3.5.8]{ME}. But  $w\in \ell_\infty(w^{-1})$ and so $\sC'w\in c_0(v^{-1})$, that is, $\lim_{n\to\infty}\frac{(\sC'w)(n)}{v(n)}=0$. Since $\sC'w=\left(\sum_{m=n}^\infty\frac{w(m)}{m}\right)_{n\in\N}$,  we obtain that
$\lim_{n\to\infty}\frac{1}{v(n)}\sum_{m=n}^\infty \frac{w(m)}{m}=0$,
that is, \eqref{e.Comp} is satisfied.

Conversely, suppose that \eqref{e.Comp} holds.  Then also \eqref{e.Co-C-1} is valid and  so $\sC\in\cL(\ell_1(v),\ell_1(w))$ by part (i) of this Proposition. Consequently, $\sC'\in \cL((\ell_\infty(w^{-1}),\ell_\infty(v^{-1}))$. Observe, for every $x\in \ell_\infty(w^{-1})$, that
\[
\frac{1}{v(n)}\left|\sum_{m=n}^\infty \frac{x_m}{m}\right|\leq \frac{1}{v(n)}\sum_{m=n}^\infty \frac{|x_m|}{w(m)}\frac{w(m)}{m}\leq \|x\|_{\infty,w^{-1}}\frac{1}{v(n)}\sum_{m=n}^\infty \frac{w(m)}{m},\quad n\in\N.
\]
Hence, by \eqref{e.Comp} it follows that $\lim_{n\to\infty}\frac{1}{v(n)}\left|\sum_{m=n}^\infty \frac{x_m}{m}\right|=0$, that is, $\sC'x\in c_0(v^{-1})$. Accordingly, $\sC'$ actually maps $\ell_\infty(w^{-1})$ into $c_0(v^{-1})$. That is, the restriction $A:=\sC'|_{c_0(w^{-1})}$,  which is continuous from $c_0(w^{-1})\su \ell_\infty(w^{-1})$ into $c_0(v^{-1})\su \ell_\infty(v^{-1})$, has the property that $A''=\sC'$ is continuous from $\ell_\infty(w^{-1})$ into $\ell_\infty(v^{-1})$ and maps $\ell_\infty(w^{-1})$ into $c_0(v^{-1})$. Accordingly, $A$ is weakly compact, \cite[Theorem 3.5.8]{ME}, and hence, also $\sC=A'$ is weakly compact from $\ell_1(v)$ into $\ell_1(w)$, \cite[Theorem 3.5.13]{ME}. Since the compact and weakly compact subsets of $\ell_1(w)\simeq \ell_1$ coincide, \cite[p.255]{ME}, it follows that $\sC$ maps $\ell_1(v)$ compactly into $\ell_1(w)$.
\end{proof}

If $v=w$ we denote $M_{v,w}$ simply by $M_w$. In the event that $M_w<\infty$, the corresponding (continuous) Ces\`aro operator $\sC\colon \ell_1(w)\to\ell_1(w)$ is denoted by $\sC^{(1,w)}$. As indicated in Section 1, we also write $\sC^{(1,w)}\in \cL(\ell_1(w))$.

The following simple fact will be used on several occasions.

\begin{lemma}\label{L-stiam2} Let $\delta>0$. Then
\[
\sum_{n=m}^\infty\frac{1}{n^{1+\delta}}\leq \frac{1}{\delta (m-1)^\delta}\leq\frac{2^\delta}{\delta  m^\delta}, \quad m\geq 2.
\]
\end{lemma}

\begin{proof} Fix $m\geq 2$. Then
\[
\sum_{n=m}^\infty\frac{1}{n^{1+\delta}}\leq \int_{m-1}^\infty\frac{1}{x^{1+\delta}}\,dx =\frac{1}{\delta (m-1)^\delta}\leq\frac{2^\delta}{\delta  m^\delta}.
\]
\end{proof}

\begin{remark}\label{R.VW}\rm Let $w=(w(n))_{n\in\N}$ be a bounded, strictly positive  weight such that $\sC^{(1,w)}\in \cL(\ell_1(w))$.

(i) Necessarily  $\alpha:=\inf_{n\in\N}w(n)=0$. Otherwise, for  $n\in\N$, we have
\[
\frac{\alpha}{w(1)}\sum_{m=1}^n\frac{1}{m}\leq \frac{1}{w(1)}\sum_{m=1}^n\frac{w(m)}{m}\leq \frac{1}{w(1)}\sum_{m=1}^\infty\frac{w(m)}{m}\leq M_w<\infty,
\]
which is impossible.

(ii)  The  condition  $\alpha:=\inf_{n\in\N}w(n)=0$, necessary  for the  continuity of $\sC$  in $\ell_1(w)$, cannot be replaced with $w\in c_0$. To see this, define $w$ by $w(n):=1$ if $n=2^k$, for $k\in\N$, and $w(n):=\frac{1}{n}$ otherwise. Surely $w\not\in c_0$. Set $a_n:=\frac{1}{w(n)}\sum_{m=n}^\infty \frac{w(m)}{m}$ for $n\in\N$. If $n=2^k$ for some  $k\in\N$, then
\[
a_{2^k}=\sum_{m=2^k}^\infty \frac{w(m)}{m}\leq \sum_{m=1}^\infty\frac{1}{m^2}+\sum_{j=k}^\infty\frac{1}{2^j}\leq \frac{\pi^2}{6}+1.
\]
Clearly, $a_1=\sum_{m=1}^\infty \frac{w(m)}{m}\leq \sum_{m=1}^\infty\frac{1}{m^2}+\sum_{j=1}^\infty\frac{1}{2^j}\leq \frac{\pi^2}{6}+1$. Finally, for fixed $k\in\N$, if $2^k<n<2^{k+1}$, then  Lemma \ref{L-stiam2} implies  that
\[
a_n\leq n\left(\sum_{m=n}^\infty\frac{1}{m^2}+\sum_{j=k+1}^\infty\frac{1}{2^j} \right)\leq n\left(\frac{2}{n}+\frac{1}{2^k}\right)\leq 4.
\]
So, $\sup_{n\in\N}a_n<\infty$, i.e.,  $\sC^{(1,w)}\in \cL(\ell_1(w))$; see Proposition \ref{P1-C}(i).

(iii) Observe that $\|e_1\|_{1,w}=w(1)$ and $\|\sC^{(1,w)}e_1\|_{1,w}=\sum_{m=1}^\infty\frac{w(m)}{m}$. So,
\[
\|\sC^{(1,w)}\|\geq \frac{\|\sC^{(1,w)}e_1\|_{1,w}}{\|e_1\|_{1,w}}=\frac{1}{w(1)}\sum_{m=1}^\infty\frac{w(m)}{m}=1+\frac{1}{w(1)}\sum_{m=2}^\infty\frac{w(m)}{m}>1.
\]
\end{remark}

Fix  $1<p<\infty$. For \textit{every} strictly  positive, \textit{decreasing} sequence  $w=(w(n))_{n\in\N}$ (i.e., $w(n+1)\leq w(n)$ for $n\in\N$)  the corresponding Ces\`aro operator $\sC^{(p,w)}$ maps $\ell_p(w)$ continuously into itself and
\begin{equation}\label{e.norm_p}
\|\sC^{(p,w)}\|\leq p',
\end{equation}
where the constant $p'=\frac{p}{p-1}$ is \textit{independent of $w$}, \cite[Proposition 2.2]{ABR}.  Example \ref{esempi-C}(ii) below shows that this is surely not the case for $p=1$. Here,
\[
\ell_p(w):=\left\{x\in \C^\N\colon \|x\|_{p,w}:=\left(\sum_{n=1}^\infty |x_n|^pw(n)\right)^{1/p}<\infty \right\}
\]
which is a Banach space for the norm $\|\cdot\|_{p,w}$ (even if $w$ is not necessarily decreasing).  Remark \ref{R.VW}(i) indicates we only need to consider decreasing weights $w\in c_0$.

\begin{examples}\label{esempi-C}\rm
(i) Fix $\gamma\in (0,1]$. Define  $w$ by $w(1):=2$ and $w(n):=\frac{1}{(\log n)^\gamma}$ for $n\geq 2$. Then  $w\downarrow 0$. Moreover, $\sC e_1=\left(\frac{1}{n}\right)_{n\in\N}$ with
\begin{equation}\label{e.serie}
\|\sC e_1\|_{1,w}=2+\sum_{n=2}^\infty\frac{1}{n(\log n)^\gamma}.
\end{equation}
By the integral  test the series \eqref{e.serie} is \textit{divergent} and so $\sC e_1\not\in\ell_1(w)$. Hence, $\sC$ does not act in $\ell_1(w)$.

(ii) For $\alpha>0$ define the decreasing weight  $w_\alpha(n):=\frac{1}{n^\alpha}$ for $n\in\N$.
Then Lemma \ref{L-stiam2} implies that
\begin{eqnarray*}
\frac{1}{w_\alpha(n)}\sum_{m=n}^\infty\frac{w_\alpha(m)}{m} &=& n^\alpha\sum_{m=n}^\infty\frac{1}{m^{\alpha+1}}\\
&\leq & \frac{n^\alpha}{\alpha(n-1)^\alpha}=\frac{1}{\alpha}\left(\frac{n}{n-1}\right)^\alpha\leq \frac{2^\alpha}{\alpha}, \quad  n\geq 2.
\end{eqnarray*}
Hence, $M_{w_\alpha}=\sup_{n\in\N}\frac{1}{w_\alpha(n)}\sum_{m=n}^\infty\frac{w_\alpha(m)}{m}\leq \frac{2^\alpha}{\alpha}$. Via Proposition \ref{P1-C}(i) we have   $\sC^{(1,w_\alpha)}\in \cL(\ell_1(w_\alpha))$. Observe that $w_\alpha\in \ell_1$ if and only if $\alpha>1$.

On the other hand, for each fixed $n\in\N$
we have
\[
\frac{1}{w_\alpha(n)}\sum_{m=n}^\infty\frac{w_\alpha(m)}{m}=n^\alpha\sum_{m=n}^\infty\frac{1}{m^{\alpha+1}}\geq n^\alpha\int_n^\infty\frac{1}{s^{\alpha+1}}\,ds=\frac{1}{\alpha}.
\]
Accordingly,
\begin{equation}\label{e.normaalfa}
\|\sC^{(1,w_\alpha)}\|=M_{w_\alpha}\geq \frac{1}{\alpha},\quad \forall \alpha>0.
\end{equation}
That is, there is no constant $K>0$ such that  $\|\sC^{(1,w)}\|\leq K$ for all decreasing weights $w\downarrow 0$ satisfying $\sC^{(1,w)}\in \cL(\ell_1(w))$.

(iii) Let now $\gamma>1$. Define  $w$ by $w(n):=\frac{1}{(\log (n+1))^\gamma}$ for $n\in\N$. Unlike in (i) above,  the integral   test reveals that now  $\sum_{n=1}^\infty\frac{1}{n(\log (n+1))^\gamma}$ is convergent. Nevertheless, $\sC$ is still not continuous from $\ell_1(w)$ into itself. To see this,  let  $g(x):=x(\log (x+1))^\gamma$, for $x>0$. Then $g$ is a strictly  increasing, positive, differentiable function in $(0,\infty)$ with $g'(x)=(\log (x+1))^\gamma+\gamma\frac{x}{x+1}(\log (x+1))^{\gamma-1}>0$ for all $x>0$. Accordingly,  $f(x):=\frac{1}{g(x)}$ is  strictly decreasing, positive, and continuous  in $(0,\infty)$. So, for fixed $n\in\N$, we have
\begin{eqnarray*}
\sum_{m=n}^\infty\frac{1}{m(\log (m+1))^\gamma}&\geq &\int_n^\infty \frac{1}{x(\log (x+1))^\gamma}\,dx \geq \int_n^\infty \frac{1}{(x+1)(\log (x+1))^\gamma}\,dx\\
&=&\frac{1}{(\gamma-1)(\log (n+1))^{\gamma-1}}.
\end{eqnarray*}
It follows that
\[
\frac{1}{w(n)}\sum_{m=n}^\infty\frac{w(m)}{m}\geq (\log(n+1))^\gamma\frac{1}{(\gamma-1)(\log (n+1))^{\gamma-1}}=\frac{\log(n+1)}{\gamma-1}.
\]
Accordingly,
$\sup_{n\in\N}\frac{1}{w(n)}\sum_{m=n}^\infty\frac{w(m)}{m}=\infty$
and so  $\sC$ fails to  map $\ell_1(w)$  into itself; see  Proposition \ref{P1-C}(i).
\end{examples}

Examples \ref{esempi-C}(i), (iii) show if $w\downarrow 0$ ``too slowly'', then $\sC$ may fail to act in $\ell_1(w)$. On the other hand, Example \ref{esempi-C}(ii)  indicates if $w\downarrow 0$ ``somewhat faster'' than in Examples \ref{esempi-C}(i), (iii) (note there that even $w_\alpha\in \ell_1$ for all $\alpha>1$), then $\sC^{(1,w)}$ may be continuous in $\ell_1(w)$. Unfortunately, no rate of decay for  $w\downarrow 0$ can be specified apriori to ensure that $\sC$ acts  in $\ell_1(w)$.

Given two bounded, strictly positive sequences $v$, $w$ satisfying $v(n)\leq w(n)$ for all $n\in\N$ we simply write $v\leq w$.

\begin{prop}\label{P.W_N} Let $v$ be any bounded, strictly positive sequence satisfying $\inf_{n\in\N}v(n)=0$.
\begin{itemize}
\item[\rm (i)] There exists a decreasing, strictly positive sequence $w\leq v$ such that $w\in c_0$ and $\sC$ does not act in $\ell_1(w)$.
\item[\rm (ii)] There exists a decreasing, strictly positive sequence $u\leq v$ such that  $\sC^{(1,u)}$is a compact operator in $\ell_1(u)$.
\end{itemize}
\end{prop}

\begin{proof} (i) Define $\varphi(n):=\min\{v(k)\colon 1\leq k\leq n\}$ for $n\in\N$. Then $\varphi$ is
strictly positive, decreasing,   satisfies $\varphi\leq v$ and  $\varphi\in c_0$.

Since $\lim_{n\to\infty}\sum_{m=k}^n\frac{1}{m}=\infty$, for all $k\in\N$, there exists
 a strictly increasing sequence $(k_j)_{j\in\N}$ in $\N$ (with $k_1:=1$)  satisfying
\begin{equation}\label{eq.magg}
\sum_{m=k_j+1}^{k_{j+1}}\frac{1}{m}>j,\quad j\in\N.
\end{equation}
Define $w(1):=1$ and $w(n):=\varphi(k_{j+1})$ for $n\in\{k_j+1,\ldots, k_{j+1}\}$ and each $j\in\N$. Since $\varphi$ is decreasing, so is $w$. In addition, for $j\in\N$ and $k_j+1\leq n\leq k_{j+1}$ we have $w(n)=\varphi(k_{j+1})\leq \varphi(n)$, that is, $w\leq \varphi\leq v$  with $w\in c_0$.
For each $j\in\N$ we have
\[
\frac{1}{w(k_j+1)}\sum_{m=k_j+1}^\infty \frac{w(m)}{m}\geq \frac{1}{w(k_j+1)}\sum_{m=k_j+1}^{k_{j+1}} \frac{w(m)}{m} \ge\sum_{m=k_j+1}^{k_{j+1}}\frac{1}{m}>j
\]
and hence,    $\sup_{n\in\N}\frac{1}{w(k_j+1)}\sum_{m=k_j+1}^\infty\frac{w(m)}{m}=\infty$. Then Proposition \ref{P1-C}(i) shows that  $\sC$ does not act  in $\ell_1(w)$.

(ii) Set $u(1):=v(1)$. Inductively, for $n\in\N$ with $u(1),\ldots, u(n)$ already specified, define
\[
u(n+1):=\min\left\{v(n+1), \frac{u(n)}{n+1}\right\}.
\]
Then $u$ satisfies $0<u\leq v$ with $u$ decreasing and $u(n+1)\leq \frac{u(n)}{n+1}$ for all $n\in\N$. Accordingly, $\lim_{n\to\infty}\frac{u(n+1)}{u(n)}=0$ and hence, by Proposition \ref{P2-C} below, we have that $\sC^{(1,u)}\in \cK(\ell_1(u))$.
\end{proof}

\begin{remark}\label{R.W_N} \rm
(i) In the statement of Proposition \ref{P.W_N} no assumption is made on $v$ as to whether or not $\sC$ acts in $\ell_1(v)$. The  behaviour exhibited in Proposition \ref{P.W_N} in relation to $\sC$ acting in $\ell_1(w)$ or not acting in $\ell_1(w)$ (even when $w\downarrow 0$) has no counterpart in the spaces $\ell_p(w)$, $1<p<\infty$. Indeed, in these spaces, for \textit{every} decreasing sequence $w$ the Ces\`aro operator $\sC$ is automatically continuous; see the discussion prior to Examples \ref{esempi-C}. The difference is that for $\ell_1(w)$ the continuity condition \eqref{e.Co-C-1} need  not respect existing  monotonicity properties of $w$.

(ii) A sequence $x=(x_n)_{n\in\N}\in \C^\N$ is  \textit{rapidly decreasing} if $(n^kx_n)_{n\in\N}\in \ell_1$ for every $k\in\N$. The space of all such sequences is denoted by $s$.
Let $v\in s$ be arbitrary. Proposition \ref{P.W_N}(ii) implies that there always exists a strictly positive weight $u\leq v$ (hence, $u\in s$) with $\sC^{(1,u)}\in \cK(\ell_1(u))$. By applying Proposition \ref{P.W_N}(i) to $u$  it follows that there exists another strictly positive sequence $w\leq u$ (hence, also $w\in s$) such that $\sC$ does \textit{not} act in $\ell_1(w)$.

(iii) The inequality \eqref{e.normaalfa}, together with \eqref{e.Comp} when $v=w$, provides a class of weights $w_\alpha \downarrow 0$, for $\alpha>0$, such that $\sC^{(1,w_\alpha)}\in \cL(\ell_1(w_\alpha))$ but, $\sC^{(1,w_\alpha)}$ \textit{fails} to be compact. \end{remark}

We now  exhibit a large class of weights $w \downarrow 0$ for which $\sC^{(1,w)}$ \textit{is} compact.

\begin{lemma}\label{L.stima} Let $r\in (0,1)$. Then
\begin{equation}\label{eq.limite}
\lim_{n\to\infty}\frac{1}{r^n}\sum_{m=n}^\infty \frac{r^m}{m}=0.
\end{equation}
\end{lemma}

\begin{proof}
Clearly \eqref{eq.limite} follows from the following inequalities
\[
\frac{1}{r^n}\sum_{m=n}^\infty\frac{r^m}{m}\leq \frac{1}{nr^n}\sum_{m=n}^\infty{r^m}=\frac{1}{nr^n}\frac{r^n}{(1-r)}=\frac{1}{n(1-r)},\quad n\in\N.
\]
\end{proof}

\begin{prop}\label{P2-C} Let $w$ be a bounded, strictly positive sequence such that  $\limsup_{n\to\infty}\frac{w(n+1)}{w(n)}=:l\in [0,1)$. Then $\sC^{(1,w)}\in \cK(\ell_1(w))$.
\end{prop}

\begin{proof} Let $r$ satisfy $l<r<1$. Then there exists $n_0\in\N$ such that $\sup_{n\geq n_0}\frac{w(n+1)}{w(n)}<r$ and hence, $w(n+1)<rw(n)$ for all $n\geq n_0$. It follows,
for a fixed $n\geq n_0$,  that
  $w(m)<r^{m-n} w(n)$ for all $m\geq n$. So, for all $n\geq n_0$, we can conclude that
\[
\frac{1}{w(n)}\sum_{m=n}^\infty\frac{w(m)}{m}\leq \frac{1}{w(n)}\sum_{m=n}^\infty \frac{r^{m-n}w(n)}{m}=\frac{1}{r^n}\sum_{m=n}\frac{r^m}{m}.
\]
Then  Lemma \ref{L.stima} shows that \eqref{e.Comp} holds, i.e., $\sC^{(1,w)}\in \cK(\ell_1(w))$.
\end{proof}

\begin{examples}\label{E.COP}\rm (i) Let $w(n):=n^\beta r^n$, for $r\in (0,1)$ and $\beta\geq 0$ fixed and for all $n\in\N$. Then $\lim_{n\to\infty}\frac{w(n+1)}{w(n)}=r\in (0,1)$.

(ii) Let $w(n)=\frac{1}{n^n}$ for $n\in\N$. Then
\[
\lim_{n\to\infty}\frac{w(n+1)}{w(n)}=\lim_{n\to\infty}\frac{1}{n+1}\left(\frac{n}{n+1}\right)^n= 0.
\]

(iii) Fix  $a>0$.  Let $w(n)=\frac{a^n}{n!}$ for all $n\in\N$. Then
\[
\lim_{n\to\infty}\frac{w(n+1)}{w(n)}=\lim_{n\to\infty}\frac{a}{n+1}= 0.
\]

(iv) Let $w$ be the positive sequence defined by $w(1):=1$ and $w(n+1):=a_nw(n)$ for $n\in\N$, where $a_{2p}:=\frac{1}{2}$ and $a_{2p-1}:=\frac{1}{p}$ for $p\in\N$. Then, for fixed $p\in\N$, we have $\frac{w(2p+1)}{w(2p)}=a_{2p}=\frac{1}{2}$ and $\frac{w(2p)}{w(2p-1)}=a_{2p-1}=\frac{1}{p}$. Accordingly, $\limsup_{n\to\infty}\frac{w(n+1)}{w(n)}=\frac{1}{2}$.

According to Proposition \ref{P2-C}, each of the weights $w$ in (i)-(iv) has the property that $\sC^{(1,w)}\in \cK(\ell_1(w))$.

(v) Fix $0<\beta<1$ and set $w_\beta(n):=e^{-n^\beta}$ for $n\in\N$. Since
\[
\lim_{n\to\infty}\frac{w_\beta(n+1)}{w_\beta(n)}=\lim_{n\to\infty}e^{n^\beta-(n+1)^\beta}=\lim_{n\to\infty}e^{-\beta/n^{1-\beta}}=1,
\]
because $n^\beta-(n+1)^\beta=n^\beta(1-[1+\frac{\beta}{n}+o(\frac{1}{n})])\simeq -\beta/n^{1-\beta}$, we see that Proposition \ref{P2-C} is not applicable. However,
\[
\frac{1}{w_\beta(n)}\sum_{m=n}^\infty\frac{w_\beta(m)}{m}=e^{n^\beta}\sum_{m=n}^\infty\frac{e^{-m^\beta}}{m^\beta m^{1-\beta}}\leq \frac{e^{n^\beta}}{n^\beta}\int_{n-1}^\infty \frac{e^{-x^\beta}}{x^{1-\beta}}\,dx
\]
as $x\mapsto \frac{e^{-x^\beta}}{x^{1-\beta}}=\frac{1}{x^{1-\beta}e^{x^\beta}}$ is decreasing in $(0,\infty)$. Since $\frac{d}{dx}\left(-\frac{1}{\beta}e^{-x^\beta}\right)=\frac{e^{-x^\beta}}{x^{1-\beta}}$, it follows that $\int_{n-1}^\infty \frac{e^{-x^\beta}}{x^{1-\beta}}\,dx=\frac{1}{\beta}e^{-(n-1)^\beta}$ and hence, that
\[
\frac{1}{w_\beta(n)}\sum_{m=n}^\infty\frac{w_\beta(m)}{m}\leq \frac{1}{\beta n^\beta}e^{n^\beta-(n-1)^\beta}\simeq \frac{e^{\beta/n^{1-\beta}}}{\beta n^\beta}.
\]
But, $\lim_{n\to\infty}\frac{e^{\beta/n^{1-\beta}}}{\beta n^\beta}=0$ and so Proposition \ref{P1-C}(ii), with $v:=w_\beta$, implies that $\sC^{(1,w_\beta)}\in \cL(\ell_1(w_\beta))$ is compact.

If $\beta=1$, then $w_\beta(n)=e^{-n}$ for $n\in\N$ and so $\lim_{n\to\infty}\frac{w_\beta(n+1)}{w_\beta(n)}=\frac{1}{e}<1$. For $\beta >1$, observe from above that
$
\lim_{n\to\infty}\frac{w_\beta(n+1)}{w_\beta(n)}=\lim_{n\to\infty}e^{-\beta n^{\beta-1}}=0$.
So, for $\beta\geq 1$ the compactness of $\sC^{(1,w_\beta)}$ does follow from Proposition \ref{P2-C}.

(vi) Fix $\gamma >1$ and set $w_\gamma(n):= e^{-\log^\gamma (n)}$ for $n\in\N$. It is shown in \cite[Remark 2.10(ii)]{ABR-9} that $\lim_{n\to\infty}\frac{w_\gamma(n+1)}{w_\gamma(n)}=1$ and so Proposition \ref{P2-C} is not applicable. However,
\begin{eqnarray*}
A_n:=\frac{1}{w_\gamma(n)}\sum_{m=n}^\infty\frac{w_\gamma(m)}{m}&=&e^{\log^\gamma (n)}\sum_{m=n}^\infty\frac{e^{-\log^\gamma (m)}}{m}\\
&\leq & e^{\log^\gamma (n)}\int_{n-1}^\infty\frac{e^{-\log^\gamma (x)}}{x}\,dx,\quad n\geq 2,
\end{eqnarray*}
because $x\mapsto \frac{e^{-\log^\gamma (v)}}{x}=\frac{1}{x \log^\gamma (x)}$ is decreasing in $(1,\infty)$. Accordingly,
\begin{eqnarray*}
A_n &\leq &  e^{\log^\gamma (n)}\int_{n-1}^\infty\frac{e^{-\log^\gamma (x)}}{x}\,dx= e^{\log^\gamma (n)}\int_{n-1}^\infty\, \frac{-1}{\gamma \log^{\gamma -1}(x)}f'(x)\,dx\\
&\leq & \frac{e^{\log^\gamma (n)}}{\gamma\log^{\gamma-1}(n-1)}\int_{n-1}^\infty (-f'(x))\,dx=\frac{e^{\log^\gamma (n)-\log^\gamma(n-1)}}{\gamma\log^{\gamma-1}(n-1)},
\end{eqnarray*}
where $f(x)=e^{-\log^\gamma (x)}$ (i.e., $f'(x)=\frac{-\gamma \log^{\gamma-1}(x) e^{-\log^\gamma(x)}}{x}$). But, for $n\geq 2$,
\[
\log^\gamma(n)-\log^{\gamma-1}(n-1)=g'(\xi_n), \ {\rm for\ some }\ \xi_n\in ((n-1),n),
\]
where $g(t):=\log^\gamma(t) $ satisfies $g'(t)=\frac{\gamma \log^{\gamma-1}(t)}{t}\to 0$ as $t\to \infty$. Hence, $0\leq \log^\gamma(n)-\log^{\gamma-1}(n-1)\leq 1$ for all $n\geq M$ and some $M\in\N$ with $M\geq 2$. It follows that
\[
A_n\leq \frac{e}{\gamma \log^{\gamma-1}(n-1)}, \quad  n\geq M,
\]
from which we can conclude that $\lim_{n\to\infty}A_n=0$, i.e., $\sC^{(1,w_\gamma)}\in \cK(\ell_1(w_\gamma))$ for all $\gamma>1$; see Proposition \ref{P1-C}(ii).
\end{examples}

\begin{remark}\label{R.NN_14}\rm Examples \ref{E.COP}(v), (vi)  also follow from the following fact.

\textit{Let $w$ be a bounded, strictly positive sequence with the property that, for every $k\in\N$ there exists $n(k)\in\N$ such that the sequence $(n^kw(n))_{n=n(k)}^\infty$ is decreasing. Then $\sC^{(1,w)}\in \cK(\ell_1(w))$.}

To see this, set $a_n:=\frac{1}{w(n)}\sum_{m=n}^\infty \frac{w(m)}{m}$ for $n\in\N$. Fix $k\in\N$. Then
\[
a_n=n^k\sum_{m=n}^\infty \frac{m^kw(m)}{n^kw(n)}\cdot\frac{1}{m^{k+1}}\leq n^k\sum_{m=n}^\infty\frac{1}{m^{k+1}},\quad n\geq n(k),
\]
because $\frac{m^kw(m)}{n^kw(n)}\leq 1$ for all $m\geq n$. But, $\sum_{m=n}^\infty\frac{1}{m^{k+1}}\leq \frac{1}{k(n-1)^k}$ (see Lemma \ref{L-stiam2}) and so $a_n\leq \frac{n^k}{k(n-1)^k}$ for $n\geq n(k)$. Since $\sup_{n\geq m(k)}\frac{n^k}{(n-1)^k}\leq 2$ for some $m(k)\geq n(k)$ it follows, for each $k\in\N$, that there exists $m(k)\in\N$ such that $a_n\leq \frac{2}{k}$ for all $n\geq m(k)$. This condition implies that $\lim_{n\to\infty}a_n=0$ and hence, via Proposition \ref{P1-C}(ii), that $\sC^{(1,w)}\in \cK(\ell_1(w))$.
\end{remark}

Let $v,\ w$ be bounded, strictly positive sequences satisfying $v\leq Aw$ for some constant $A>0$. Then the natural inclusion $\ell_1(w)\su \ell_1(v)$ is continuous because
$\|x\|_{1,v}\leq A\|x\|_{1,w}$, for  $x\in \ell_1(w)$.
Suppose that $\sC^{(1,w)}\in \cL(\ell_1(w))$. Then $\sC\colon \ell_1(w)\to \ell_1(v)$ is also continuous with $\|\sC\|\leq A\|\sC^{(1,w)}\|$. According to Proposition \ref{P.W_N}, $\sC$ need not have an $\ell_1(v)$-valued continuous linear extension from $\ell_1(w)$ to $\ell_1(v)$. Similarly, if $\sC^{(1,w)}\in \cK(\ell_1(w))$ and $\sC^{(1,v)}\in \cL(\ell_1(v))$, then  $\sC^{(1,v)}$  need not be  compact. The following explicit examples illustrate these features.

\begin{example}\label{E.Comparison}\rm (i) Select a strictly increasing sequence $1=:k_1<k_2<\ldots $ in $\N$ satisfying $k_{j+1}>2k_j$ for each $j\in\N$ and $\lim_{j\to\infty}\sum_{m=1+k_j}^{k_{j+1}}\frac{1}{m}=\infty$ (eg., $k_j:=j^j$). Set $v(1):=1$ and, for each $j\in\N$, define $v(n):=\frac{1}{2^j(k_{j+1}-k_j)}$ for $k_j+1\leq n\leq k_{j+1}$. For fixed $j\in\N$ it follows that
\[
\frac{1}{v(k_j+1)}\sum_{m=k_j+1}^\infty\frac{v(m)}{m}\geq \frac{1}{v(k_j+1)}\sum_{m=k_j+1}^{k_{j+1}}\frac{v(m)}{m} \ge \sum_{m=k_j+1}^{k_{j+1}}\frac{1}{m}
\]
and hence, $\sup_{j\in\N}\frac{1}{v(k_j+1)}\sum_{m=k_j+1}^\infty\frac{v(m)}{m}=\infty$. Proposition \ref{P1-C}(i) implies that $\sC$ does \textit{not} act  in $\ell_1(v)$.

On the other hand, define $w(n):=\frac{2}{n}$ for $n\in\N$. Given $n\geq 2$, select $j\in\N$ such that $k_j+1\leq n\leq k_{j+1}$. Then
\[
\frac{v(n)}{w(n)}=\frac{n}{2\cdot 2^j(k_{j+1}-k_j)}\leq \frac{k_{j+1}}{2(k_{j+1}-k_j)}=\frac{1}{2(1-\frac{k_j}{k_{j+1}})}<1
\]
and so $v\leq w$. Moreover, $\sC^{(1,w)}\in \cL(\ell_1(w))$; see Examples \ref{esempi-C}(ii).

(ii) Define the decreasing sequence $v$ by $v(1)=v(2):=1$ and
\[
v(n):= \frac{1}{2^i2^{(i+1)2^{i+1}}}, \ \mbox{ for }\ 2^i+1\leq n\leq 2^{i+1}\ \mbox{ and }\ i\in\N,
\]
and the sequence $w:=(\frac{1}{n^{n+1}})_{n\in\N}$. Given $n\geq 3$ select $i\in\N$ such that $2^i+1\leq n\leq 2^{i+1}$. Then
\[
\frac{v(n)}{w(n)}=\frac{n^nn}{2^i2^{(i+1)2^{i+1}}}\leq \frac{(2^{i+1})^{2^{i+1}}2^{i+1}}{2^i2^{(i+1)2^{i+1}}}=2.
\]
Since $\frac{v(1)}{w(1)}=1$ and $\frac{v(2)}{w(2)}=8$, it follows that $v\leq 8w$. In particular, $v\in s$. According to Proposition \ref{P2-C} (as $\lim_{n\to\infty}\frac{w(n+1)}{w(n)}=0$)  $\sC^{(1,w)}$ is compact. On the other hand, $\sC^{(1,v)}$ is continuous (see Fact 2 in Example \ref{ex.nocompact} below) but \textit{not} compact (see Fact 3 in Example \ref{ex.nocompact} below).
\end{example}

We now present a positive comparison result where difficulties such as those observed in Example \ref{E.Comparison} do not arise.

\begin{prop}\label{P.Comparison} Let $v,\ w$ be bounded, strictly positive sequences such that $(\frac{v(n)}{w(n)})_{n=n_0}^\infty$ is a decreasing sequence for some  $n_0\in\N$.
\begin{itemize}
\item[\rm (i)] If $\sC^{(1,w)}\in \cL(\ell_1(w))$, then also $\sC^{(1,v)}\in \cL(\ell_1(v))$.
\item[\rm (ii)] If $\sC^{(1,w)}\in \cK(\ell_1(w))$, then also $\sC^{(1,v)}\in \cK(\ell_1(v))$.
\end{itemize}
\end{prop}

\begin{proof} (i) Define $\alpha_n:=\frac{v(n)}{w(n)}$ for $n\in\N$ in which case $\alpha_{n}\geq \alpha_{n+1}$ for all $n\geq n_0$. Proposition \ref{P1-C}(i) implies that $M_w<\infty$; see \eqref{e.Co-C-1}. Moreover,
\begin{equation}\label{eq.comparison}
\frac{1}{v(n)}\sum_{m=n}^\infty \frac{v(m)}{m}=\frac{1}{w(n)}\sum_{m=n}^\infty \frac{\alpha_m}{\alpha_n}\cdot\frac{w(m)}{m}\leq \frac{1}{w(n)}\sum_{m=n}^\infty\frac{w(m)}{m}\leq M_w,
\end{equation}
for all $n\geq n_0$. In particular, via \eqref{eq.comparison},
\[
A:=\sum_{k=1}^\infty \frac{v(k)}{k}=\sum_{k=1}^{n_0-1} \frac{v(k)}{k}+v(n_0)\cdot\frac{1}{v(n_0)}\sum_{k=n_0}^\infty \frac{v(k)}{k}<\infty.
\]
For each $n\in\{1,\ldots,n_0-1\}$ it follows that $\frac{1}{v(n)}\sum_{m=n}^\infty \frac{v(m)}{m}\leq \frac{A}{v(n)}$ and hence,
\[
M_v=\sup_{n\in\N}\frac{1}{v(n)}\sum_{m=n}^\infty \frac{v(m)}{m}\leq \max\left\{M_w,\max\left\{\frac{A}{v(1)},\ldots,\frac{A}{v(n_0-1)}\right\}\right\}<\infty.
\]
Accordingly, $\sC^{(1,v)}\in \cL(\ell_1(v))$.

(ii) Let $n_0$ be as in the statement of the proposition. Let $\epsilon>0$. Since $\sC^{(1,w)}$ is compact, there exists $n_1(\epsilon)>n_0$ such that
\[
\frac{1}{w(n)}\sum_{m=n}^\infty\frac{w(m)}{m}<\epsilon,\quad n\geq n_1(\epsilon);
\]
see \eqref{e.Comp}. It then follows from \eqref{eq.comparison} that also
\[
\frac{1}{v(n)}\sum_{m=n}^\infty \frac{v(m)}{m}<\epsilon,\quad n\geq n_1(\epsilon)>n_0.
\]
Accordingly, $\sC^{(1,v)}$ is also compact; see Proposition \ref{P1-C}(ii).
\end{proof}

\begin{example}\label{E.ConNOComp}\rm (i) For $w:=(\frac{1}{n^\alpha})_{n\in\N}$ with $\alpha>0$, Examples \ref{esempi-C}(ii) shows that $\sC^{(1,w)}\in \cL(\ell_1(w))$. Define $v(n):=\frac{1}{n^\alpha\log ^\beta(n+1)}$, $n\in\N$, with $\beta>0$. Then $\frac{v}{w}$ is a decreasing sequence and so Proposition \ref{P.Comparison}(i) implies that $\sC^{(1,v)}\in \cL(\ell_1(v))$.

(ii) Let $w(n):=\frac{1}{n\log ^\beta(n+1)}$, $n\in\N$, with $\beta>1$, in which case $\sC^{(1,w)}\in \cL(\ell_1(w))$ by part (i). Also, via Examples \ref{esempi-C}(ii), $v:=(\frac{1}{n^2})_{n\in\N}$ satisfies $\sC^{(1,v)}\in \cL(\ell_1(v))$. Consider the sequence $\frac{v}{w}=(\frac{\log^\beta(n+1)}{n})_{n\in\N}$. The derivative of the function $f(x):=\frac{\log^\beta(x+1)}{x}$ for $x\geq 1$ is given by
\[
f'(x)=\frac{(\beta x-(x+1)\log (x+1))\log^{\beta-1}(x+1)}{x^2(x+1)}
\]
and hence, $f$ is decreasing on $((e^\beta-1),\infty)$. So there exists $n_0\in\N$ such that $(\frac{v(n)}{w(n)})_{n=n_0}^\infty$ is decreasing. Since $\sC^{(1,v)}$ is not compact (by Remark \ref{R.W_N}(iii)), it follows from Proposition \ref{P.Comparison}(ii) that $\sC^{(1,w)}$ also fails to be compact.
\end{example}

\begin{remark}\label{R.Comparison_1}\rm Let $v,\ w$ be bounded, strictly positive sequences satisfying $A_1v\leq w\leq A_2 v$ for positive constants $A_1$, $A_2$. Then $\ell_1(v)$ and $\ell_1(w)$ are equal as vector spaces and the norms $\|\cdot\|_{1,v}$ and  $\|\cdot\|_{1,w}$ are equivalent. Accordingly, $\sC^{(1,w)}\in \cL(\ell_1(w))$ (resp. $\cK(\ell_1(w)$) if and only if $\sC^{(1,v)}\in \cL(\ell_1(v))$ (resp. $\cK(\ell_1(v)$). For instance, let $v=(\frac{1}{n^\alpha})_{n\in\N}$ with $\alpha>0$. Consider \textit{any} bounded, strictly positive sequence $\varphi$ satisfying $\gamma:=\inf_{n\in\N}\varphi(n)>0$. Then $w:=(\varphi(n)v(n))_{n\in\N}$ satisfies $\gamma v\leq w \leq \|\varphi\|_\infty v$. Via Examples \ref{esempi-C}(ii), $\sC^{(1,v)}\in \cL(\ell_1(v))$ and so also $\sC^{(1,w)}\in \cL(\ell_1(w))$. Remark \ref{R.W_N}(iii) shows that $\sC^{(1,v)}$ is not compact and hence, also $\sC^{(1,w)}$ fails to be compact. Or, suppose that $u\leq v$. Then $v\leq u+v\leq 2v$ and so  $\sC^{(1,u+v)}$ is continuous (resp. compact) if and only if  $\sC^{(1,v)}$ is continuous (resp. compact).
\end{remark}

\section{Spectrum of $\sC^{(1,w)}$}

The aim of this section is to provide some detailed knowledge of the spectrum of $\sC^{(1,w)}$. For $1<p<\infty$ it is known for \textit{every strictly positive, decreasing} weight $w$ that the spectrum of $\sC^{(p,w)}\in \cL(\ell_p(w))$ satisfies
\[
\sigma(\sC^{(p,w)})\su \{\lambda\in \C\colon |\lambda|\leq p'\}
\]
with $p'=\frac{p}{p-1}$ a constant \textit{independent} of $w$; see \eqref{e.norm_p} above and \cite[Theorem 3.3(i)]{ABR}. It will be shown, for $p=1$, that no such constant (independent of $w$) exists; see Example \ref{es.enne}. The spectrum of $\sC^{(1,w)}$ is characterized  in Theorem \ref{t:sp-dec-N}. Further properties of $\sigma(\sC^{(1,w)})$ are exhibited in Proposition \ref{t:sp-dec}. Whenever $\sC^{(1,w)}$ is a compact operator, a complete description of $\sigma(\sC^{(1,w)})$ is given in Proposition \ref{P_SR}. Several relevant examples are presented.

We begin by recalling the following known fact; see e.g. \cite[Proposition 4.1]{ABR-7}, \cite[Propositions 4.3 and 4.4]{ABR-1}. For convenience of notation we set $\Sigma:=\{\frac{1}{m}\colon m\in\N\}$ and $\Sigma_0:=\Sigma\cup\{0\}$. Recall that $\C^\N$ is a Fr\'echet space for the lc-topology of coordinatewise convergence.

\begin{lemma}\label{L1}
\begin{itemize}
\item[\rm (i)] The spectrum $\sigma(\sC,\C^\N)=\sigma_{pt}(\sC,\C^\N)=\Sigma$.
\item[\rm (ii)] Fix $m\in\N$. Let $x^{(m)}:=(x_{n}^{(m)})_n\in \C^\N$ where $x_{n}^{(m)}:=0$ for $n\in\{1,\ldots, m-1\}$, $x_{m}^{(m)}:=1$ and $x_{n}^{(m)}:=\frac{(n-1)!}{(m-1)!(n-m)!}$ for $n>m$. Then the $1$-dimensional eigenspace of $\frac{1}{m}$ is given by
\[
\Ker \left(\frac{1}{m}I-\sC\right)={\rm span}\{x^{(m)}\}\su\C^\N.
\]
\end{itemize}
\end{lemma}

\begin{remark}\label{R.1n}\rm For $\lambda=1$,  the corresponding eigenvector for $\sC\colon \C^\N\to\C^\N$ is the constant vector $\mathbf{1}:= (1)_{n\in\N}$. Accordingly, if $w$ is any bounded, strictly positive weight such that $\sC^{(1,w)}\in \cL(\ell_1(w))$, then ${1}\in \sigma_{pt}(\sC^{(1,w)})$ if and only if $\mathbf{1}\in \ell_1(w)$, i.e., if and only if $w\in \ell_1$.
\end{remark}

The following inequalities, \cite[Lemma 3.2]{ABR}, \cite[Lemma 7]{R}, will be needed.

\begin{lemma}\label{l-prod}
 {\rm (i)} Let $\lambda\in \C\setminus\Sigma_0$ and set $\alpha:={\rm Re}\left(\frac{1}{\lambda}\right)$. Then there exist constants $d>0$ and $D>0$ {\rm (}depending on $\alpha${\rm )} such that
\begin{equation}\label{e.prod}
\frac{d}{n^\alpha}\leq \prod_{k=1}^n\left|1-\frac{1}{k\lambda}\right|\leq \frac{D}{n^\alpha}, \quad n\in\N.
\end{equation}

{\rm (ii)} For each $m\in\N$ we have  that
\begin{equation}\label{e.fact}
\frac{(n-1)!}{(n-m)!}\simeq  n^{m-1},\quad \mbox{for all large } n\in\N.
\end{equation}
\end{lemma}

For every bounded, strictly positive weight $w=(w(n))_{n\in\N}$ recall  that
\begin{equation}\label{e.err_w}
R_w:=\{t\in\R\colon \sum_{n=1}^\infty n^tw(n)<\infty\}.
\end{equation}
In case $R_w\not=\R$ we define $t_0:=\sup R_w$.

\begin{prop}\label{P2-Sp} Let $w=(w(n))_{n=1}^\infty$ be a bounded, strictly positive sequence. The following conditions are equivalent.
\begin{itemize}
\item[\rm (i)] $(n^mw(n))_n\in \ell_1$ for all $m\in\N$.
\item[\rm (ii)] $w\in s$.
\item[\rm (iii)] $R_w=\R$.
\end{itemize}
If, in addition, $\sC^{(1,w)}\in \cL(\ell_1(w)$, then {\rm (i)-(iii)} are equivalent to
\begin{itemize}
\item[\rm (iv)] $\Sigma\su \sigma_{pt}(\sC^{(1,w)})$.
\end{itemize}
\end{prop}

\begin{proof} (i)$\Leftrightarrow$(ii) follows from the definition of the space $s$.

(i)$\Leftrightarrow$(iii) follows from the definition of $R_w$; see \eqref{e.err_w}.

Assume now that $\sC^{(1,w)}\in \cL(\ell_1(w))$.

(iv)$\Rightarrow$(i) Fix $m\in\N$. Then $\frac{1}{m+1}\in \sigma_{pt}(\sC^{(1,w)})\su \sigma_{pt}(\sC, \C^\N)$ with $ x^{(m+1)}$  as  its eigenvector in $\C^\N$; see Lemma \ref{L1}. So, necessarily $x^{(m+1)}\in \ell_1(w)$, i.e., $(w(n)x^{(m+1)}_n)_{n\in\N}\in \ell_1$. But, this happens only if $(n^mw(n))_n\in \ell_1$; see \eqref{e.fact}.

(i)$\Rightarrow$(iv) Fix $m\in\N$. Then $(n^{m-1}w(n))_n\in \ell_1$ and so  the sequence $(w(n)x^{(m)}_n)_{n\in\N}\in \ell_1$, i.e., $x^{(m)}\in \ell_1(w)$, where $x^{(m)}$ is as in Lemma \ref{L1}. As $x^{(m)}$ is an eigenvector corresponding to the eigenvalue $\frac{1}{m}$ for $\sC$ acting on $\C^\N$, it follows that $\frac{1}{m}$ is also an eigenvalue  for $\sC^{(1,w)}$.
\end{proof}

Given a strictly positive, bounded sequence $w=(w(n))_{n\in\N}$, recall that $S_w(1):=\{s\in\R\colon \sup_{n\in\N} \frac{1}{n^sw(n)}<\infty\}$.
In case $S_w(1)\not=\emptyset$ we defined $s_1:=\inf S_w(1)$.
Since $\frac{1}{n^sw(n)}\geq \frac{1}{n^s\|w\|_\infty}$, for $n\in\N$, it follows that $s\not\in S_w(1)$ for every $s<0$ and hence, $S_w(1)\su[0,\infty)$. Accordingly, $s_1\geq 0$. If $w(n)\geq \alpha$ for all $n\in\N$ and some $\alpha>0$, then $\sC$ does not act in $\ell_1(w)$; see Remark \ref{R.VW}(i). So, we restrict our attention to weights $w$ with $\inf_{n\in\N}w(n)=0$. In this case $\frac{1}{w}\not\in \ell_\infty$. Hence, if $S_w(1)\not=\emptyset$ and $s\in S_w(1)$, then necessarily $s>0$ with $\frac{M}{n^s}\leq w(n)$ for some $M>0$ and all $n\in\N$. It follows that $[s,\infty)\su S_w(1)$. So, whenever $S_w(1)\not=\emptyset$ (with $\frac{1}{w}\not\in \ell_\infty$) we can conclude that $S_w(1)$ is an interval of the form $[s_1,\infty)$ or $(s_1,\infty)$ with $s_1\geq 0$.

Concerning  $R_w$ (see \eqref{e.err_w}), whenever $R_w\not=\R$ the quantity $t_0$ is finite with $t_0\geq -1$ and $R_w=(-\infty, t_0)$ or $R_w=(-\infty, t_0]$. Moreover, $R_w=\emptyset$ is impossible as $\sum_{n=1}^\infty n^tw(n)\leq \|w\|_\infty\sum_{n=1}^\infty n^t<\infty$ whenever $t<-1$.

\begin{prop}\label{P.S_R} Let $w$ be a bounded, strictly positive sequence.
\begin{itemize}
\item[\rm (i)] If $S_w(1)\not=\emptyset$, then $t_0\leq s_1$. In particular, $R_w\not=\R$.
\item[\rm (ii)] If $R_w\not=\R$, then $S_w(1)\su [t_0,\infty)$.
\item[\rm (iii)] If $w\in s$, then $S_w(1)=\emptyset$.
\end{itemize}
\end{prop}

\begin{proof} (i) Fix any $s>s_1$. Then, for some $M>0$, we have $\frac{1}{n^sw(n)}\leq M$ for all $n\in\N$ and hence, $n^sw(n)\geq \frac{1}{M}$ for $n\in\N$. Accordingly, $(n^sw(n))_{n\in\N}\not\in\ell_1$, i.e., $s\not\in R_w$. This implies that $(s_1,\infty)\su \R\setminus R_w$, i.e., $t_0\leq s_1$.

(ii) Fix any $t<t_0$ in which case $\lim_{n\to\infty}t^nw(n)=0$. Hence, there exists $K\in\N$ such that $n^t\leq \frac{1}{w(n)}$ for $n\geq K$. So, for any $s\in \R$ we have  (as $\frac{1}{n^s}>0$ for $n\in\N$) that $\frac{1}{n^sw(n)}\geq\frac{n^t}{n^s}$ for all $n\geq K$. Hence, if $s<t$, then $(\frac{1}{n^sw(n)})_{n\in\N}\not\in\ell_\infty$ and so $s\not\in S_w(1)$. This implies that $S_w(1)\su [t_0,\infty)$.

(iii) Suppose $r<0$. Since $w\in c_0$, there exists $L\in\N$ such that $\frac{1}{w(n)}\geq 1$ for $n\geq L$. Hence, $\frac{1}{n^rw(n)}\geq\frac{1}{n^r}$ for all $n\geq L$ and so $r\not\in S_w(1)$. For $r\geq 0$ fixed, set $m:=1+[r]$. Since $\sum_{n=1}^\infty n^mw(n)<\infty$ (see Proposition \ref{P2-Sp}), there is $J\in\N$ such that $\frac{1}{w(n)}\geq n^m$ for $n\geq J$ and hence, $\frac{1}{n^rw(n)}\geq \frac{n^m}{n^r}$ for $n\geq J$, that is, $r\not\in S_w(1)$.
\end{proof}

\begin{remark}\label{R.2n}\rm (i) The converse of Proposition \ref{P.S_R}(iii) is not valid. Indeed, let $w=(w(n))_{n\in\N}$ be the strictly positive weight with $w\downarrow 0$ as given in \cite[Remark 3.2]{ABR-9}. It is shown there that there exists a strictly increasing sequence $(n(k))_{k\in\N}$ in $\N$ with the  property: for each $t\in\R$ we have
\[
\frac{1}{(n(k))^tw(n(k))}\geq k, \quad  k>t.
\]
Hence, $t\not\in S_w(1)$ and so $S_w(1)=\emptyset$. It is also shown in \cite{ABR-9} that $w\not\in s$.

(ii) If $v,\ w$ are bounded, strictly positive sequences with $w\leq v$, then $\frac{1}{v}\leq \frac{1}{w}$ from which it follows that $S_1(w)\su  S_1(v)$. Hence, $\inf S_1(v)\leq \inf S_1(w)$. Also, it is clear from \eqref{e.err_w} that $R_v\su R_w$ and so $\sup R_v\leq \sup R_w$.
\end{remark}

We now come to the main results of this section. The following result characterizes the spectrum of $\sC^{(1,w)}$.

\begin{theorem}\label{t:sp-dec-N} Let $w=(w(n))_{n\in\N}$ be a  bounded, strictly positive sequence such  that $\sC^{(1,w)}\in \cL(\ell_1(w))$.

{\rm (i)} The following inclusions hold:
\begin{equation}
\label{sp-point} \Sigma\su \Sigma_0\su \sigma(\sC^{(1,w)}).
\end{equation}

{\rm (ii)} Let $\lambda\not\in \Sigma_0$ and set $\alpha:={\rm Re}\left(\frac{1}{\lambda}\right)$. Then $\lambda\in \rho(\sC^{(1,w)})$ if and only if
\begin{equation}\label{e.spe_car}
\sup_{m\in\N}\frac{1}{m^\alpha w(m)}\sum_{n=m+1}^{\infty}\frac{w(n)}{n^{1-\alpha}}<\infty.
\end{equation}

{\rm (iii)}
Suppose that  $R_w\not=\R$, i.e., $t_0<\infty$. Then
\begin{equation}\label{e:point}
\left\{\frac{1}{m}\colon m\in\N,\ (m-1)\in R_w\right\}=\sigma_{pt}(\sC^{(1,w)})\su\left\{\frac{1}{m}\colon m\in\N,\ 1\leq m\leq t_0+1\right\}.
\end{equation}

In particular, $\sigma_{pt}(\sC^{(1,w)})$ is a finite subset of $\Sigma$ {\rm (}possibly empty{\rm )}.

If $R_w=\R$, then
\begin{equation}\label{e:pointN}
\sigma_{pt}(\sC^{(1,w)})=\Sigma.
\end{equation}
\end{theorem}

\begin{proof}The proof is via a series of steps.

(i)
The dual operator $A:=(\sC^{(1,w)})'\in \cL(\ell_{\infty}(w^{-1}))$ is given by
\begin{equation}\label{e:dualeop}
Ay=\left(\sum_{k=n}^\infty\frac{y_k}{k}\right)_{n\in\N}, \quad y=(y_n)_{n\in\N}\in \ell_{\infty}(w^{-1}).
\end{equation}

\textbf{Step 1.} $0\not \in \sigma_{pt}(A)$.

If $Ay=0$, for some $y\in\ell_{\infty}(w^{-1})$,  then  $z_n:=\sum_{k=n}^\infty\frac{y_k}{k}=0$ for all $n\in\N$. Hence,  $y_n={n}(z_n-z_{n+1})=0$, for $n\in\N$, and so $A$ is injective.

\textbf{Step 2.} $\Sigma\su \sigma_{pt}(A)$.

Let $\lambda\in \Sigma$, i.e., $\lambda=\frac{1}{m}$ for some $m\in\N$. Via \eqref{e:aut} below, the non-zero vector $y=(y_n)_{n\in\N}$ defined via $y_1\in\C\setminus\{0\}$ arbitrary,  $y_n:=y_1\prod_{k=1}^{n-1}\left(1-\frac{1}{\lambda k}\right)$ for $1<n\leq m$ and $y_n:=0$ for $n> m$, which  clearly belongs to $\ell_{\infty}(w^{-1})$, satisfies $Ay=\lambda y$.

\textbf{Step 3.} $\Sigma_0\su \sigma(\sC^{(1,w)})$.

For every $T\in \cL(X)$ with $X$ a Banach space, we have $\sigma_{pt}(T')\su \sigma(T)$, \cite[p.581]{DSI}, with $\sigma(T)$  closed in $\C$.  By Step 2 we then have $\Sigma_0\su \sigma(\sC^{(1,w)})$.

(ii) \textbf{Step 4.} \textit{Fix $\lambda\not\in \Sigma_0$. Then $\lambda\in \rho(\sC^{(1,w)})$ if and only if \eqref{e.spe_car} holds}.

 To verify this we argue in a similar way as in  \cite{ABR} or \cite{CR}. We recall the formula for $(\sC-\lambda I)^{-1}\colon \C^\N\to \C^\N$ whenever $\lambda\not \in \Sigma_0$, \cite[p.266]{R}. Namely, for $n\in\N$, the $n$-th row of the matrix for $(\sC-\lambda I)^{-1}$ has the entries
\[
\frac{-1}{n\lambda^2\prod_{k=m}^n\left(1-\frac{1}{\lambda k}\right)},\quad 1\leq m<n,
\]
\[
\frac{n}{1-n\lambda}=\frac{1}{\frac{1}{n}-\lambda}, \quad m=n,
\]
and all the other entries in row $n$ are equal to $0$. So,
we can write
\begin{equation}\label{e:dec}
(\sC-\lambda I)^{-1}=D_\lambda-\frac{1}{\lambda^2}E_\lambda,
\end{equation}
where the diagonal operator $D_\lambda=(d_{nm})_{n,m\in\N}$ is given by $d_{nn}:=\frac{1}{\frac{1}{n}-\lambda}$ and $d_{nm}:=0$ if $n\not=m$. The operator $E_\lambda=(e_{nm})_{n,m\in\N}$ is then the lower triangular matrix with $e_{1m}=0$ for all $m\in\N$ and for every $n\geq 2$, with $e_{nm}:=\frac{1}{n\prod_{k=m}^n\left(1-\frac{1}{\lambda k}\right)}$ if $1\leq m<n$ and $e_{nm}:=0$ if $m\geq n$.

As $\lambda\not \in \Sigma_0$, we have $d(\lambda):={\rm dist}(\lambda, \Sigma_0)>0$ and $|d_{nn}|\leq \frac{1}{d(\lambda)}$ for $n\in\N$. Hence, for every $x\in \ell_1(w)$, it follows that
\[
\|D_\lambda(x)\|_{1,w}=\sum_{n=1}^\infty |d_{nn}x_n|w(n)\leq \frac{1}{d(\lambda)}\sum_{n=1}^\infty|x_n|w(n)=\frac{1}{d(\lambda)}\|x\|_{1,w}.
\]
This means that $D_\lambda\in \cL(\ell_1(w))$. So, by \eqref{e:dec} it remains to show that $E_\lambda\in \cL(\ell_1(w))$ if  and only if  \eqref{e.spe_car} is satisfied for $\alpha:={\rm Re}\left(\frac{1}{\lambda }\right)$.

To this end, we note that $E_\lambda\in \cL(\ell_1(w))$ if  and only if the operator $\tilde{E}_\lambda\colon \C^\N\to \C^\N$ given by $\tilde{E}_\lambda=\Phi_wE\Phi_w^{-1}$, i.e.,
\[
(\tilde{E}_\lambda(x))_n=w(n)\sum_{m=1}^{n-1}\frac{e_{nm}}{w(m)}x_m, \quad x\in \C^\N,\ n\in\N,
\]
defines a continuous linear operator on $\ell_1$ (see the comments prior to Lemma \ref{L_U}). So, the claim is that $\tilde{E}_\lambda\in \cL(\ell_1)$ if and only if  \eqref{e.spe_car} is satisfied for $\alpha:={\rm Re}\left(\frac{1}{\lambda }\right)$.
To establish this claim,  observe that
\eqref{e.prod} implies
\begin{eqnarray}\label{e.Boun}
& & \frac{D^{-1}}{n^{1-\alpha}}\leq |e_{n1}|\leq \frac{d^{-1}}{n^{1-\alpha}},\quad n\geq 2,\nonumber\\
& &\frac{d'D^{-1}}{n^{1-\alpha}m^\alpha}\leq  |e_{nm}|\leq \frac{d^{-1}D'}{n^{1-\alpha}m^\alpha},\quad 2\leq m<n,
\end{eqnarray}
for some constants $d'>0$ and $D'>0$ depending on $\lambda$.

Suppose first that $\tilde{E}_\lambda\in \cL(\ell_1)$. Then Lemma \ref{L_U} implies that
\[
\sup_{m\in\N}\frac{1}{w(m)}\sum_{n=m+1}^{\infty}{w(n)|e_{nm}|}<\infty.
\]
By \eqref{e.Boun} we have $\sup_{m\in\N}\frac{1}{m^\alpha w(m)}\sum_{n=m+1}^{\infty}\frac{w(n)}{n^{1-\alpha}}<\infty$, i.e., \eqref{e.spe_car} is satisfied.

Conversely, if \eqref{e.spe_car} is satisfied, then
$\sup_{m\in\N}\frac{1}{m^\alpha w(m)}\sum_{n=m+1}^{\infty}\frac{w(n)}{n^{1-\alpha}}<\infty$. By \eqref{e.Boun} this implies that also
$\sup_{m\in\N}\frac{1}{w(m)}\sum_{n=m+1}^{\infty}{w(n)|e_{nm}|}<\infty$. Therefore, via Lemma \ref{L_U}, we can conclude that $\tilde{E}_\lambda\in \cL(\ell_1)$. The claim is proved.

 The proof of part (ii) is thereby complete.

(iii) Suppose first that $R_w\not =\R$.

\textbf{Step 5.} \textit{Both the equality and the inclusion in  \eqref{e:point} are valid.}

The proof of Step 5 is a routine adaption of the proof of Step 7 in the proof of Theorem 3.3 in \cite{ABR}; just substitute $p=1$ there. In particular, it follows that $\frac{1}{m}\in \sigma_{pt}(\sC^{(1,w)})$ if and only if $(m-1)\in R_w$.

The previous observation also allows us to adapt the argument of Step 8 in the proof of Theorem 3.3 in \cite{ABR} to establish the following (final)

\textbf{Step 6.}\textit{ Assume  that  $R_w=\R$. Then \eqref{e:pointN} is valid.}
\end{proof}

\begin{remark}\label{R.3n}\rm (i) Step 1 in the proof of Theorem \ref{t:sp-dec-N} implies that $\sC^{(1,w)}$ has dense range, i.e., $0$ belongs to the continuous spectrum of $\sC^{(1,w)}$.

(ii) It is clear from \eqref{e:point} that if $\frac{1}{M}\in \sigma_{pt}(\sC^{(1,w)})$ for some $M\in\N$, then also $\frac{1}{m}\in \sigma_{pt}(\sC^{(1,w)})$ for all $m\in\{1,\ldots, M\}$.

(iii) It can happen that $\sigma_{pt}(\sC^{(1,w)})=\emptyset$; see Example \ref{es.enne} below. In view of part (i) and Remark \ref{R.1n} this is equivalent to $1\not\in \sigma_{pt}(\sC^{(1,w)})$, i.e., $w\not\in \ell_1$.

(iv) Suppose $v,\ w$ are bounded, strictly positive sequences such that $\sC^{(1,w)}\in \cL(\ell_1(w))$ and $(\frac{v(n)}{w(n)})_{n=n_0}^\infty$ is  decreasing  for some $n_0\in\N$. Proposition \ref{P.Comparison}(i) implies that $\sC^{(1,v)}\in \cL(\ell_1(v))$. Let $\lambda\in \rho(\sC^{(1,w)})\setminus\Sigma_0$, i.e., \eqref{e.spe_car} holds for $\alpha:={\rm Re}\left(\frac{1}{\lambda}\right)$. Setting $\alpha_n:=\frac{v(n)}{w(n)}$ for $n\in\N$ it follows, for  $m\geq n_0$, that
\[
\frac{1}{m^\alpha v(m)}\sum_{n=m+1}^\infty \frac{v(n)}{n^{1-\alpha}}=\frac{1}{m^\alpha w(m)}\sum_{n=m+1}^\infty \frac{\alpha_m}{\alpha_n}\cdot\frac{w(n)}{n^{1-\alpha}}\leq \frac{1}{m^\alpha w(m)}\sum_{n=m+1}^\infty \frac{w(n)}{n^{1-\alpha}}.
\]
Hence, \eqref{e.spe_car} implies that $\gamma:=\sup_{m\geq n_0} \frac{1}{m^\alpha v(m)}\sum_{n=m+1}^\infty \frac{v(n)}{n^{1-\alpha}}<\infty$. Set $\delta:=\max\{\frac{1}{w(m)}\colon 1\leq m<n_0\}$. Then, for $m\in\{1,\ldots, n_0-1\}$, we have
\begin{eqnarray*}
& &\frac{1}{m^\alpha v(m)}\sum_{n=m+1}^\infty \frac{v(n)}{n^{1-\alpha}}=\frac{1}{m^\alpha v(m)}\sum_{k=m+1}^{n_0} \frac{v(k)}{k^{1-\alpha}}+\frac{n_0^\alpha v(n_0)}{m^\alpha v(m)}\cdot \frac{1}{n_0^\alpha v(n_0)}\sum_{n=n_0+1}^\infty \frac{v(n)}{n^{1-\alpha}}\\
& & \quad \leq \frac{1}{ v(m)}\sum_{k=m+1}^{n_0} \frac{v(k)}{k^{1-\alpha}}+\left(\frac{n_0}{m}\right)^\alpha\frac{\gamma v(n_0)}{v(m)}\leq \delta \sum_{k=1}^{n_0} \frac{v(k)}{k^{1-\alpha}}+n_0^\alpha\delta\gamma v(n_0).
\end{eqnarray*}
Accordingly, $\sup_{m\in\N} \frac{1}{m^\alpha v(m)}\sum_{n=m+1}^\infty \frac{v(n)}{n^{1-\alpha}}<\infty$ and so Theorem \ref{t:sp-dec-N}(ii), applied to $v$, shows that $\lambda\in \rho(\sC^{(1,v)})\setminus\Sigma_0$, that is
\[
\sigma(\sC^{(1,v)})\su \sigma(\sC^{(1,v)})\cup \Sigma_0\su \sigma(\sC^{(1,w)})\cup \Sigma_0.
\]

Of course, if bounded, strictly positive sequences $v$ and $w$ satisfy $A_1 v\leq w\leq A_2 v$ for positive constants $A_1,\ A_2$ and $\sC^{(1,w)}\in \cL(\ell_1(w))$, then Remark \ref{R.Comparison_1} implies that $\sigma(\sC^{(1,w)})= \sigma(\sC^{(1,v)})$. As an application, for fixed $\alpha>1$ consider the sequence $w$ given by $w(1)=w(2)=1$ and $w(n):=\frac{1}{i^\alpha 2^{i-1}}$ for $2^i+1\leq n\leq 2^{i+1}$ and $i\in\N$. Define $v(n):=\frac{1}{n\log^\alpha(n+1)}$ for $n\in\N$. Then
\begin{equation}\label{eq.comp_1}
A_1 v\leq w\leq A_2 v
\end{equation}
for positive constants $A_1,\ A_2$. To establish \eqref{eq.comp_1}, fix $n\geq 3$ and select $i\in\N$ such that $2^i+1\leq n\leq 2^{i+1}$. Then
\[
\frac{w(n)}{v(n)}=\frac{n\log^\alpha(n+1)}{i^\alpha 2^{i-1}}\geq \frac{2^i\log^\alpha(2^i)}{i^\alpha2^{i-1}}=2\log^\alpha 2.
\]
Since $\frac{w(1)}{v(1)}=\log ^\alpha 2$ and $\frac{w(2)}{v(2)}=2\log^\alpha 3$ it follows, with $A_1=2\log^\alpha 3$ that the first inequality in \eqref{eq.comp_1} is satisfied. Concerning the other inequality in \eqref{eq.comp_1} observe, still with $n$ and $i$ as above, that
\[
\frac{w(n)}{v(n)}=\frac{n\log^\alpha(n+1)}{i^\alpha 2^{i-1}}\leq \frac{2^{i+1}\log^\alpha (2^{i+2})}{i^\alpha 2^{i-1}}=4\left(\frac{i+2}{i}\right)^\alpha \log^\alpha 2,\ n\geq 3.
\]
Since $\lim_{i\to\infty }\frac{i+2}{i}=1$, there exists $A_2>0$ such that $w\leq A_2v$. This establishes \eqref{eq.comp_1}. It is shown in Example \ref{E.ConNOComp}(ii) that $\sC^{(1,v)}\in \cL(\ell_1(v))$. According to \eqref{eq.comp_1} and Remark \ref{R.Comparison_1} also $\sC^{(1,w)}\in \cL(\ell_1(w))$. The above discussion and \eqref{eq.comp_1} then imply that
\begin{equation}\label{eq.comp_2}
\sigma_{pt}(\sC^{(1,v)})=\sigma_{pt}(\sC^{(1,w)})\ \mbox{ and }\ \sigma(\sC^{(1,v)})=\sigma(\sC^{(1,w)}).
\end{equation}
Combining Fact 8 of Example \ref{E.an} below with \eqref{eq.comp_2} yields
\[
\sigma_{pt}(\sC^{(1,v)})=\{1\}\ \mbox{ and }\ \sigma(\sC^{(1,v)})=\left\{\lambda\in\C\colon \left|\lambda-\frac{1}{2}\right|\leq \frac{1}{2}\right\}.
\]
\end{remark}

It was noted above that always $s_1\geq 0$. Note, for $w=(\frac{1}{\log (n+1)})_{n\in\N}$, that $s_1=0$ and  $w\downarrow 0$.  But, $\sC$ does not
act in  $\ell_1(w)$;  see Example \ref{esempi-C}(i). For weights $w$  with $\sC^{(1,w)}\in \cL(\ell_1(w))$ this phenomenon cannot occur.

\begin{prop}\label{t:sp-dec} Let  $w$ be a bounded, strictly positive sequence such that $\sC^{(1,w)}\in \cL(\ell_1(w))$ and $S_w(1)\not=\emptyset$.

{\rm (i)} It is necessarily the case that $s_1>0$.

{\rm (ii)} For the dual operator $(\sC^{(1,w)})'\in \cL(\ell_\infty(w^{-1}))$ of $\sC^{(1,w)}$ we have
\begin{equation}\label{e:spdD}
\left\{\lambda\in \C\colon \left|\lambda-\frac{1}{2s_1}\right|< \frac{1}{2s_1}\right\}\cup \Sigma \su \sigma_{pt}((\sC^{(1,w)})')
\end{equation}

and

\begin{equation}\label{e:spd-2}
\sigma_{pt}((\sC^{(1,w)})')\setminus \Sigma\su \left\{\lambda\in \C\colon \left|\lambda-\frac{1}{2s_1}\right|\leq \frac{1}{2s_1}\right\}.
\end{equation}

For the Ces\`aro operator $\sC^{(1,w)}$ itself we have
\begin{equation}\label{e:spd}
\left\{\lambda\in \C\colon \left|\lambda-\frac{1}{2s_1}\right|\leq \frac{1}{2s_1}\right\}\cup \Sigma\su \sigma(\sC^{(1,w)}).
\end{equation}
\end{prop}

\begin{proof} (i) Suppose that $s_1:=\inf S_w(1)=0$. Fix any $s>0$. Then $w(n)\geq \frac{c(s)}{n^s}$ for some constant $c(s)>0$ and all $n\in\N$. Hence, $\sum_{n=1}^\infty \frac{w(n)}{n^{1-s}}\geq c(s) \sum_{n=1}^\infty \frac{1}{n}$ which shows that $\sum_{n=1}^\infty \frac{w(n)}{n^{1-s}}$ \textit{diverges} (for every $s>0$).

Fix $s>0$ with $s\not\in\Sigma$ and  set $\lambda:=\frac{1}{s}\in\R$. By the previous paragraph $\sum_{n=1}^\infty \frac{w(n)}{n^{1-s}}$ diverges.
   Theorem \ref{t:sp-dec-N}(ii) implies (put $m=1$ in \eqref{e.spe_car}) that $\lambda\not\in \rho(\sC^{(1,w)})$, i.e., $\lambda\in \sigma(\sC^{(1,w)})$. So,
   the unbounded set $\{\frac{1}{s}\colon s>0\}\setminus \Sigma$ is contained in $\sigma(\sC^{(1,w)})$;  impossible. Hence,  $s_1>0$.

(ii)  We proceed by a series of steps. Denote again by $A\in \cL(\ell_\infty(v))$ the dual operator of $\sC^{(1,w)}$, where $v:=w^{-1}$.

\textbf{Step 1.} $\left\{\lambda\in \C\colon \left|\lambda-\frac{1}{2s_1}\right|< \frac{1}{2s_1}\right\}\su \sigma_{pt}(A)$.

Let $\lambda\in \C\setminus\{0\}$. Then  $Ay=\lambda y$ is satisfied for some non-zero $y\in\ell_{\infty}(v)$   if and only if $\lambda y_n=\sum_{k=n}^\infty \frac{y_k}{k}$ for all $n\in\N$. This yields, for every $n\in\N$, that $\lambda (y_n-y_{n+1})=\frac{y_n}{n}$ and so
$y_{n+1}=\left(1-\frac{1}{\lambda n}\right)y_n$.
It follows that
\begin{equation}\label{e:aut}
y_{n+1}=y_1\prod_{k=1}^n \left(1-\frac{1}{\lambda k}\right),\quad  n\in\N,
\end{equation}
with $y_1\not=0$. In particular,
 each eigenvalue of $A$ is simple.

Let now $\lambda\in\C\setminus\Sigma$ satisfy $\left|\lambda-\frac{1}{2s_1}\right|<\frac{1}{2s_1}$ (equivalently,  $\alpha:={\rm Re}\left(\frac{1}{\lambda}\right)>{s_1}$); note that $\lambda\not=0$. For such a $\lambda$ the vector $y=(y_n)_{n\in\N}\in \C^\N$ defined by \eqref{e:aut} actually belongs to $\ell_{\infty}(v)$. Indeed,  via  Lemma \ref{l-prod}(i) there exists $c=c(\lambda)>0$ such  that
\[
\prod_{k=1}^n \left|1-\frac{1}{\lambda k}\right|\leq cn^{-{\rm Re}(1/\lambda)},\quad n\in\N.
\]
It then follows from \eqref{e:aut} that
\[
|y_n|w(n)^{-1} =|y_1|w(n)^{-1} \prod_{k=1}^n \left|1-\frac{1}{\lambda k}\right|\leq  c|y_1|n^{-{\rm Re}(1/\lambda)}w(n)^{-1},
\]
where the sequence $( n^{-{\rm Re}(1/\lambda)}w(n)^{-1})_{n\in\N}$ is bounded because ${\rm Re}(1/\lambda)\in S_w(1)$. That is, $y\in \ell_{\infty}(v)$. Hence, $\lambda\in \sigma_{pt}(A)$.

\textbf{Step 2.} $\sigma_{pt}(A)\setminus \Sigma_0\su \left\{\lambda\in \C\colon \left|\lambda-\frac{1}{2s_1}\right|\leq \frac{1}{2s_1}\right\}$.

Fix $\lambda\in  \sigma_{pt}(A)\setminus \Sigma_0$. According to \eqref{e.prod}
 there is $\beta=\beta(\lambda)>0$ such that
\begin{equation}\label{e:stima-n}
\prod_{k=1}^n \left|1-\frac{1}{\lambda k}\right|\geq \beta\cdot n^{-{\rm Re}(1/\lambda)},\quad n\in\N.
\end{equation}
But, as argued in Step 1 (for any $y_1\in \C\setminus\{0\}$) the eigenvector $y=(y_n)_{n\in\N}$ corresponding to the eigenvalue $\lambda$ of $A$, which necessarily belongs to $\ell_{\infty}(v)$, i.e., $\sup_{n\in\N} |y_n|w(n)^{-1}<\infty$, is given by \eqref{e:aut}. Then \eqref{e:stima-n} implies that also
 $\sup_{n\in\N} \frac{1}{n^{{\rm Re}(1/\lambda)}w(n)}<\infty$ (i.e.,  ${\rm Re}\left(\frac{1}{\lambda}\right)\in S_w(1)$) and so ${\rm Re}\left(\frac{1}{\lambda}\right)\geq s_1$, that is,  $\lambda\in \left\{\mu\in \C\colon \left|\mu-\frac{1}{2s_1}\right|\leq \frac{1}{2s_1}\right\}$.

It is clear that Steps 1-2 above, together with Steps 1 and 2 in the proof of Theorem \ref{t:sp-dec-N}, establish the two containments in \eqref{e:spdD} and  \eqref{e:spd-2}.

For  $T\in \cL(X)$, with $X$ a Banach space,  $\sigma_{pt}(T')\su \sigma(T)$, \cite[p.581]{DSI}, with $\sigma(T)$  closed in $\C$. So, \eqref{e:spd} follows from \eqref{e:spdD}.
\end{proof}

\begin{remark}\label{R.W_NN}\rm The converse of Proposition \ref{t:sp-dec}(i) is not valid. Indeed, for \textit{every} $\alpha>0$ the exists a weight $w\downarrow 0$ with $s_1=\alpha$ but, $\sC$ does \textit{not} act in $\ell_1(w)$. To see this let $(k_j)_{j\in\N}\su \N$ (with $k_1:=1$)  be a strictly increasing sequence satisfying \eqref{eq.magg}. Define $w(1):=1$ and $w(n):=\frac{1}{(k_j+1)^\alpha}$ for each $j\in\N$ and $k_j+1\leq n\leq k_{j+1}$. For $s:=\alpha$ observe, if $j\in\N$ and  $k_j+1\leq n\leq k_{j+1}$, then
\[
\frac{1}{n^sw(n)}=\frac{(k_j+1)^\alpha}{n^\alpha}\leq \frac{(k_j+1)^\alpha}{(k_j+1)^\alpha}=1.
\]
Hence, $\sup_{n\in\N}\frac{1}{n^sw(n)}<\infty$ and so $\alpha\in S_w(1)$, i.e., $[\alpha,\infty)\su S_w(1)$. On the other hand, for each $j\in\N$ and $n:=k_j+1$, we have for each $s<\alpha$ that
\[
\sup_{j\in\N}\frac{1}{n^sw(n)}=\sup_{j\in\N} \frac{(k_j+1)^\alpha}{(k_j+1)^s}=\infty.
\]
It follows that $\sup_{n\in\N}\frac{1}{n^sw(n)}=\infty$ and so $s\not\in S_w(1)$. Hence, we have established that  $S_w(1)=[\alpha,\infty)$, i.e., $s_1=\alpha$. Arguing as in the proof of Proposition \ref{P.W_N}(i) it follows that $\sup_{j\in\N}\frac{1}{w(k_j+1)}\sum_{m=k_j+1}^\infty \frac{w(m)}{m}=\infty$ and so, via Proposition \ref{P1-C}(i), $\sC$ does \textit{not} act in $\ell_1(w)$.
\end{remark}

According to Proposition \ref{P.S_R}(iii), such weights $w$ (with $s_1 > 0$) cannot exist in $s$.

The following result should be compared with \cite[Proposition 4.1]{ABR}.

\begin{prop}\label{P_SR} Let $w$ be a bounded, strictly positive weight such that $\sC^{(1,w)}\in \cK(\ell_1(w))$. Then the following properties hold.
\begin{itemize}
\item[\rm (i)] $\sigma_{pt}(\sC^{(1,w)})=\Sigma$ and $\sigma(\sC^{(1,w)})=\Sigma_0$.
\item[\rm (ii)] $w\in s$.
\end{itemize}
\end{prop}

\begin{proof} (i) It is clear that $0\not\in \sigma_{pt}(\sC^{(1,w)})$ as $\sC\colon \C^\N\to\C^\N$ is injective. The compactness of $\sC^{(1,w)}$ then implies that $\sigma_{pt}(\sC^{(1,w)})=\sigma(\sC^{(1,w)})\setminus\{0\}$, \cite[Theorem 3.4.23]{ME}. Moreover, Lemma \ref{L1} reveals that also $\sigma_{pt}(\sC^{(1,w)})\su \sigma_{pt}(\sC, \C)=\Sigma$. The two previous facts, together with Theorem  \ref{t:sp-dec-N}(i), imply the validity of the two equalities in (i).

(ii) By Theorem \ref{t:sp-dec-N}(iii) we must have $R_w=\R$. Otherwise, $t_0$ is finite and so \eqref{e:point} implies that $\sigma_{pt}(\sC^{(1,w)})$ is a finite set. This is a contradiction to part (i). So, $R_w=\R$ and hence, $w\in s$; see Proposition \ref{P2-Sp}.
\end{proof}

\begin{remark}\label{R.4n}\rm (i) If $w$ is  a bounded, strictly positive weight such that $\sC^{(1,w)}\in \cL(\ell_1(w))$ and $S_w(1)\not=\emptyset$, then \eqref{e:spd}
implies that $\sC^{(1,w)} \notin \cK (\ell_1(w))$.

(ii) Recall, for $T\in \cL(X)$ with $X$ a Banach space, that $T$ is \textit{power bounded} if
$\sup_{n\in\N}\|T^n\|<\infty  $.
If $w$ is as in part (i) and $0<s_1<1$, then $\frac{1}{s_1}>1$. It follows from \eqref{e:spd} and the spectral mapping theorem that $[0, \frac{1}{s_1^n}]\su \sigma((\sC^{(1,w)})^n)$ for all $n\in\N$. Then the spectral radius inequality implies  that $\frac{1}{s_1^n}\leq \|(\sC^{(1,w)})^n\|$ for all $n\in\N$. Accordingly, $\sC^{(1,w)}$ cannot be power bounded. Since also $\frac{\|(\sC^{(1,w)})^n\|}{n}\geq \frac{(s_1^{-1})^n}{n}$ for $n\in\N$, it follows from the Principle of Uniform Boundedness that $\left\{\frac{(\sC^{(1,w)})^n}{n}\right\}_{n\in\N}$ cannot converge in $\cL(\ell_1(w))$. In particular, $\sC^{(1,w)}$ cannot be mean ergodic; see the discussion prior to Lemma \ref{L.Mean-0} below.
\end{remark}

It is time for some relevant examples.

\begin{example}\label{es.enne}\rm (i) Let $w_\alpha(n)=\frac{1}{n^\alpha}$, $n\in\N$, for any fixed $\alpha>0$. According to Example \ref{esempi-C}(ii) we have  $\sC^{(1,w_\alpha)}\in \cL(\ell_1(w_\alpha))$. It is routine to check that  $S_{w_\alpha}(1)=[\alpha, \infty)$ and hence, $s_1=\alpha>0$. The claim is that
\begin{equation}\label{e.spettro}
\left\{\lambda\in \C\colon \left|\lambda-\frac{1}{2\alpha}\right|\leq \frac{1}{2\alpha}\right\}\cup \Sigma= \sigma(\sC^{(1,w_\alpha)}).
\end{equation}
Indeed, according  to \eqref{e:spd}  we have
\[
\left\{\lambda\in \C\colon \left|\lambda-\frac{1}{2\alpha}\right|\leq \frac{1}{2\alpha}\right\}\cup \Sigma\su \sigma(\sC^{(1,w_\alpha)}).
\]
To establish the reverse inclusion, fix $\lambda\in\C\setminus\Sigma$ such that $\left|\lambda-\frac{1}{2\alpha}\right|> \frac{1}{2\alpha}$. We show that $\lambda\in \rho(\sC^{(1,w_\alpha)})$. To  this effect, set $\beta:={\rm Re}\left(\frac{1}{\lambda}\right)$. Then $\beta<\alpha$, that is, $(\alpha-\beta)>0$.
Lemma \ref{L-stiam2} implies, for every $m\in\N$,  that
\[
\sum_{n=m+1}^\infty\frac{w_\alpha(n)}{n^{1-\beta}}=\sum_{n=m+1}^\infty\frac{1}{n^{1+\alpha-\beta}}\leq \frac{2^{\alpha-\beta}}{(\alpha-\beta)(m+1)^{\alpha-\beta}}
\]
and so
\[
\sup_{m\in\N}\frac{1}{m^\beta w_\alpha(m)}\sum_{n=m+1}^\infty\frac{w_\alpha(n)}{n^{1-\beta}}\leq \sup_{m\in\N}\frac{2^{\alpha-\beta}m^{\alpha-\beta}}{(\alpha-\beta)(m+1)^{\alpha-\beta}}\leq \frac{2^{\alpha-\beta}}{(\alpha-\beta)}.
\]
Hence, Theorem \ref{t:sp-dec-N}(ii) implies that $\lambda\in \sigma(\sC^{(1,w_\alpha)})$, as claimed.

For $0<\alpha<1$ we see that $0<s_1<1$ and so Remark \ref{R.4n}(ii) implies that $\sC^{(1,w_\alpha)}$ is not power bounded.

It is clear from \eqref{e.spettro}, as alluded to in the beginning of this section,  that there is no constant $K>0$ such that
\[
\sigma(\sC^{(1,w)})\su \{\lambda\in\C\colon |\lambda|\leq K\}
\]
for all strictly positive, decreasing weights $w$ satisfying $\sC^{(1,w)}\in \cL(\ell_1(w))$.

It is routine to check that $R_{w_\alpha}=(-\infty, (\alpha-1))$ and so $t_0=(\alpha-1)$. For $0<\alpha\leq 1$ it follows that $(m-1)\not\in R_{w_\alpha}$ for all $m\in\N$, that is,  $\sigma_{pt}(\sC^{(1,w_\alpha)})=\emptyset$; see \eqref{e:point}. This also follows from the fact that $w_\alpha\not\in\ell_1$; see Remark \ref{R.1n} and Remark \ref{R.3n}(iii). For $\alpha>1$ (in which case $w_\alpha\in \ell_1$), it follows from \eqref{e:point} that
\[
\sigma_{pt}(\sC^{(p,w_\alpha)})=\left\{\frac{1}{m}\colon m\in\N,\ 1\leq m<\alpha\right\}.
\]

(ii) Let $w(n):=1$ if $n=2^k$, for $k\in\N$, and $w(n):=\frac{1}{n}$ otherwise. It is shown in Remark \ref{R.VW}(ii) that $\sC^{(1,w)}\in \cL(\ell_1(w))$. The claim is that
\begin{equation}\label{e.W_N}
\sigma(\sC^{(1,w)})=\left\{\lambda\in\C\colon \left|\lambda-\frac{1}{2}\right|\leq \frac{1}{2}\right\}.
\end{equation}
Since $w\not\in\ell_1$, Remark \ref{R.3n}(iii) shows that $\sigma_{pt}(\sC^{(1,w)})=\emptyset$. Let $s>0$. Then $\frac{1}{n^sw(n)}=\frac{1}{2^{ks}}$ if $n=2^k$, for $k\in\N$, and $\frac{1}{n^sw(n)}=\frac{n}{n^s}$ if $2^k<n<2^{k+1}$ for some $k\in\N$. Accordingly, $\sup_{n\in\N}\frac{1}{n^sw(n)}<\infty$ if and only if $s\geq 1$, i.e., $S_w(1)=[1,\infty)$ with $s_1=1$. According to \eqref{e:spd}  we have that $\sigma(\sC^{(1,w)})$ is contained in the right-side of \eqref{e.W_N}. To establish the reverse inclusion, fix $\lambda\in\C\setminus\Sigma$ such that $|\lambda-\frac{1}{2}|> \frac{1}{2}$.   To verify $\lambda\in \rho(\sC^{(1,w)})$,  set $\alpha:={\rm Re}\left(\frac{1}{\lambda}\right)$ and $r:=\frac{1}{2^{1-\alpha}}$.
Then  $1-\alpha>0$ and $r\in (0,1)$.

For $m=2^k$, $k\in\N$, it follows from Lemma \ref{L-stiam2} that
\begin{eqnarray*}
& & \frac{1}{m^\alpha w(m)}\sum_{n=m+1}^\infty\frac{w(n)}{n^{1-\alpha}}=\frac{1}{2^{\alpha k}}\sum_{n=m+1}^\infty\frac{w(n)}{n^{1-\alpha}}\\
& &\leq \frac{1}{2^{\alpha k}}\left(\sum_{n=2^k}^\infty\frac{1}{nn^{1-\alpha}}+\sum_{j=k}^\infty\frac{1}{(2^j)^{1-\alpha}}\right)=\frac{1}{2^{\alpha k}}\left(\sum_{n=2^k}^\infty\frac{1}{n^{1+(1-\alpha)}}+\sum_{j=k}^\infty r^j\right)\\
& &\leq \frac{1}{2^{\alpha k}}\left(\frac{2^{1-\alpha}}{(1-\alpha)2^{k(1-\alpha)}}+\frac{r^k}{(1-r)}\right)=\frac{2^{1-\alpha}}{(1-\alpha)2^{k}}+\frac{1}{(1-r)2^k}\\
& &\leq \frac{2^{1-\alpha}}{(1-\alpha)}+\frac{1}{(1-r)}.
\end{eqnarray*}
On the other hand, if $2^k<m<2^{k+1}$ for some $k\in\N$, then
\begin{eqnarray*}
& & \frac{1}{m^\alpha w(m)}\sum_{n=m+1}^\infty\frac{w(n)}{n^{1-\alpha}}\leq m^{1-\alpha}\sum_{n=2^k+1}^\infty \frac{w(n)}{n^{1-\alpha}}\\
 & & \leq 2^{(1-\alpha)(k+1)} \left(\sum_{n=2^k}^\infty\frac{1}{nn^{1-\alpha}}+\sum_{j=k}^\infty\frac{1}{(2^j)^{1-\alpha}}\right) \\
 & & \leq  2^{(1-\alpha)k}2^{(1-\alpha)}\left(\frac{2^{1-\alpha}}{(1-\alpha)2^{k(1-\alpha)}}+\frac{1}{(1-r)2^{(1-\alpha)k}}\right)\\
& &=\frac{2^{2(1-\alpha)}}{(1-\alpha)}+\frac{2^{1-\alpha}}{(1-r)}  .
\end{eqnarray*}
So, there exists a constant $c$ (depending on $r $ and  $\alpha$) such that
\[
\sup_{m\in\N}\frac{1}{m^\alpha w(m)}\sum_{n=m+1}^\infty\frac{w(n)}{n^{1-\alpha}}\leq c  .
\]
Hence, Theorem \ref{t:sp-dec-N}(ii) implies that $\lambda\in\rho(\sC^{(1,w)})$. So,  the right-side of \eqref{e.W_N} is contained in $\sigma(\sC^{(1,w)})$.
This completes the argument  establishing \eqref{e.W_N}. Finally, \eqref{e.W_N} implies that $\sC^{(1,w)}$ is \textit{not} compact.
\end{example}

Before presenting the next example we record the following simple fact.

\begin{lemma}\label{L.N} There exists a constant $c>0$ such  that
\begin{equation}\label{e.lemN}
c\leq \sum_{j=2^i+1}^{2^{i+1}}\frac{1}{j}\leq 1, \quad  i\in\N.
\end{equation}
\end{lemma}

\begin{proof} Fix $i\in\N$. Then
\[
\sum_{j=2^i+1}^{2^{i+1}}\frac{1}{j}\geq \int_{2^i+1}^{2^{i+1}+1}\frac{dx}{x}=\log\left(\frac{2^{i+1}+1}{2^i+1}\right)
\]
with $\lim_{i\to\infty}\log\left(\frac{2^{i+1}+1}{2^i+1}\right)=\log (2)$. So, there exists $K\in\N$ with $\sum_{j=2^i+1}^{2^{i+1}}\frac{1}{j}\geq \frac{\log (2)}{2}$, for all $i\geq K$, which implies the existence of $c>0$ satisfying the first inequality in \eqref{e.lemN}.
The second inequality in \eqref{e.lemN} follows from
\[
\sum_{j=2^i+1}^{2^{i+1}}\frac{1}{j}\leq \sum_{j=2^i+1}^{2^{i+1}}\frac{1}{2^{i}+1}=\frac{2^i}{2^i+1}\leq 1,\quad  i\in\N.
\]
\end{proof}

Proposition \ref{P_SR} states if $\sC^{(1,w)}\in \cK(\ell_1(w))$, then necessarily $w\in s$. The following example shows that the converse is false.

\begin{example}\label{ex.nocompact}\rm There exists a strictly positive, decreasing weight $w\in s$ such that $\sC^{(1,w)}\in \cL(\ell_1(w))$, its spectra are given by
\[
\sigma(\sC^{(1,w)})=\Sigma_0 \ \ \mbox{and}\ \ \sigma_{pt}(\sC^{(1,w)}))=\Sigma,
\]
but $\sC^{(1,w)}$ fails to be compact.

Define the decreasing sequence $w=(w(n))_{n\in\N}$ by $w(1)=w(2)=1$ and
\[
w(n):=\frac{1}{2^i2^{(i+1)2^{i+1}}}, \quad \mbox{for } 2^i+1\leq n\leq 2^{i+1}\mbox{ and } i\in\N.
\]

\textit{Fact 1. The weight $w\in s$}.

Since the sequence $(\frac{1}{n^{n+1}})_{n\in\N}$ clearly belongs to $s$  and $w(n)\leq \frac{8}{n^{n+1}}$ for $n\in\N$ (see Example \ref{E.Comparison}(ii)), it follows that also $w\in s$.


\textit{Fact 2. The operator  $\sC^{(1,w)}\in \cL(\ell_1(w))$}.

Fix $m\in\N$ with $m\geq 3$. Now choose $i\in\N$ such that $2^i+1\leq m\leq 2^{i+1}$. Using the fact, for each $k\in\N$, that $\frac{1}{n}\leq \frac{1}{2^k+1}$ whenever $2^k+1\leq n\leq 2^{k+1}$, that each sum of the form $\sum_{n=2^k+1}^{2^{k+1}}(\ldots)$ has $2^k$ terms, and that $\frac{2^k}{2^k+1}\leq 1$, it follows that
\begin{eqnarray*}
\sum_{n=m}^\infty\frac{w(n)}{n}&\leq & \sum_{n=2^i+1}^\infty\frac{w(n)}{n}=\sum_{k=i}^\infty w(2^k+1)\sum_{n=2^k+1}^{2^{k+1}}\frac{1}{n}\\
&\leq & \sum_{k=i}^\infty w(2^k+1)\frac{2^k}{2^k+1}\leq \sum_{k=i}^\infty w(2^k+1).
\end{eqnarray*}
Due to the definition of $w(2^k+1)$ for $k\geq i$ we can conclude that
\begin{eqnarray*}
\sum_{n=m}^\infty\frac{w(n)}{n}&\leq & \sum_{k=i}^\infty\frac{1}{2^k2^{(k+1)2^{k+1}}}<\sum_{k=i}^\infty\frac{1}{2^k2^{(i+1)2^{i+1}}}\\
&=& \frac{1}{2^{(i+1)2^{i+1}}}\sum_{k=i}^\infty\frac{1}{2^k}=\frac{1}{2^{(i+1)2^{i+1}}}\cdot \frac{1}{2^{i-1}}.
\end{eqnarray*}
Since $\frac{1}{w(m)}=\frac{1}{w(2^i+1)}=2^i2^{(i+1)2^{i+1}}$, the previous inequality implies that
\[
\frac{1}{w(m)}\sum_{n=m}^\infty\frac{w(n)}{n}\leq 2.
\]
Accordingly, $\sup_{m\geq 3}\frac{1}{w(m)}\sum_{n=m}^\infty\frac{w(n)}{n}\leq 2$. Moreover, both $\frac{1}{w(1)}\sum_{n=1}^\infty\frac{w(n)}{n}\leq \sum_{n=1}^\infty w(n)<\infty$ and $\frac{1}{w(2)}\sum_{n=2}^\infty\frac{w(n)}{n}\leq \sum_{n=1}^\infty w(n)<\infty$. So, by \eqref{e.Co-C-1} of Proposition \ref{P1-C}(i) we can conclude that $\sC^{(1,w)}\in\cL(\ell_1(w))$.

\textit{Fact 3. The operator $\sC^{(1,w)}$ is not compact.}

By \eqref{e.Comp} of Proposition \ref{P1-C}(ii) we need to verify that $\left(\frac{1}{w(m)}\sum_{n=m}^\infty\frac{w(n)}{n}\right)_{m\in\N}$ does \textit{not} converge to $0$. Fix $i\in\N$. Then, for $m:=2^i+1$, we have
\[
\frac{1}{w(2^i+1)}\sum_{n=2^i+1}^\infty\frac{w(n)}{n}\geq \frac{1}{w(2^i+1)}\sum_{n=2^i+1}^{2^{i+1}}\frac{w(n)}{n}=\sum_{n=2^i+1}^{2^{i+1}}\frac{1}{n}\geq c
\]
with $c>0$ as in Lemma \ref{L.N}. Since $\left(\frac{1}{w(2^i+1)}\sum_{n=2^i+1}^\infty\frac{w(n)}{n}\right)_{i\in\N}$ is a subsequence of $\left(\frac{1}{w(m)}\sum_{n=m}^\infty\frac{w(n)}{n}\right)_{m\in\N}$, we are done.

\textit{Fact 4. The spectra are given by $\sigma(\sC^{(1,w)})=\Sigma_0$ and $\sigma_{pt}(\sC^{(1,w)})=\Sigma$.}

According to Fact 1 above we have $R_w=\R$ (see Proposition \ref{P2-Sp}) and so Theorem \ref{t:sp-dec-N}(iii) implies that $\sigma_{pt}(\sC^{(1,w)})=\Sigma$.

To verify that $\sigma(\sC^{(1,w)})=\Sigma_0$ we need to show that every $\lambda\not\in \Sigma_0$ belongs to $\rho(\sC^{(1,w)})$. This is achieved by considering the two possible cases. Namely, when $|\lambda-\frac{1}{2}|\leq \frac{1}{2}$ (equivalent to $\alpha:={\rm Re}\left(\frac{1}{\lambda}\right)$ satisfying $\alpha\geq 1$) and when $|\lambda-\frac{1}{2}|> \frac{1}{2}$ (equivalent to $\alpha<1$).

Case (1). Let $\alpha \geq 1$ (i.e., $(\alpha-1)\geq 0$). Then $\lambda\in \rho(\sC^{(1,w)})$.

Because $w\in s$ it is clear that $\frac{1}{1^\alpha w(1)}\sum_{n=2}^\infty \frac{w(n)}{n^{(1-\alpha)}}=\sum_{n=2}^\infty n^{\alpha-1}w(n)<\infty$ and also that $\frac{1}{2^\alpha w(2)}\sum_{n=3}^\infty \frac{w(n)}{n^{(1-\alpha)}}=\frac{1}{2^\alpha}\sum_{n=3}^\infty n^{\alpha-1}w(n)<\infty$. So, fix $m\in\N$ with $m\geq 3$. Now select $i\in\N$ with $2^i+1\leq m\leq 2^{i+1}$ in which case
\[
\sum_{n=m+1}^\infty \frac{w(n)}{n^{(1-\alpha)}}\leq \sum_{n=2^i+1}^\infty n^{\alpha-1}w(n)=\sum_{k=i}^\infty w(2^k+1)\sum_{n=2^k+1}^{2^{k+1}}n^{\alpha-1}.
\]
Since $n^{\alpha-1}\leq (2^{k+1})^{\alpha-1}$ for $2^k+1\leq n\leq 2^{k+1}$, with  $2^k$ terms, we have
\begin{eqnarray*}
& & w(2^k+1)\sum_{n=2^k+1}^{2^{k+1}}n^{\alpha-1} \leq  w(2^k+1)2^k\cdot (2^{k+1})^{\alpha-1}=\frac{1}{2^k2^{(k+1)2^{k+1}}}\cdot 2^k\cdot (2^{k+1})^{\alpha-1}\\
& &\ \ =\frac{1}{2^{(k+1)2^{k+1}}}\cdot\left(\frac{2^{\alpha-1}}{2^\alpha}\right)^{k+1}=\frac{1}{2^{(k+1)2^{k+1}}}\cdot\left(\frac{1}{2}\right)^{k+1}.
\end{eqnarray*}
It follows that
\[
\sum_{n=m+1}^\infty \frac{w(n)}{n^{(1-\alpha)}}\leq \sum_{k=i}^\infty \frac{1}{2^{(k+1)(2^{k+1}-\alpha)}}\cdot\left(\frac{1}{2}\right)^{k+1}.
\]
But, for all $k\geq i$ we have $\frac{1}{2^{(k+1)(2^{k+1}-\alpha)}}\leq \frac{1}{2^{(i+1)(2^{i+1}-\alpha)}}$ and so
\begin{equation}\label{e.copno}
\sum_{n=m+1}^\infty \frac{w(n)}{n^{(1-\alpha)}}\leq \frac{1}{2^{(i+1)(2^{i+1}-\alpha)}}\sum_{k=i}^\infty \left(\frac{1}{2}\right)^{k+1}=\frac{2^{(i+1)\alpha}}{2^i\cdot 2^{(i+1)2^{i+1}}}.
\end{equation}
Using $\frac{1}{w(m)}=2^i\cdot 2^{(i+1)2^{i+1}}$ and  $\frac{1}{m^\alpha}\leq \frac{1}{(2^i+1)^\alpha}$ it follows from \eqref{e.copno} that
\begin{eqnarray*}
& & \frac{1}{m^\alpha w(m)}\sum_{n=m+1}^\infty \frac{w(n)}{n^{(1-\alpha)}}\leq \frac{1}{(2^i+1)^\alpha}\cdot 2^i 2^{(i+1)2^{i+1}}\cdot \frac{2^{(i+1)\alpha}}{2^i\cdot 2^{(i+1)2^{i+1}}}\\
& & \quad =2^{\alpha}\left(\frac{2^i}{2^i+1}\right)^\alpha\leq 2^{\alpha}<\infty.
\end{eqnarray*}
Hence, the condition \eqref{e.spe_car} in Theorem \ref{t:sp-dec-N}(ii) is satisfied, i.e., $\lambda\in\rho(\sC^{(1,w)})$.

Case (2). Let $\alpha<1$. Then  $\lambda\in\rho(\sC^{(1,w)})$.

Because $w\in s\su \ell_1$ it is clear that $\frac{1}{1^\alpha w(1)}\sum_{n=2}^\infty \frac{w(n)}{n^{(1-\alpha)}}\leq \sum_{n=2}^\infty w(n)<\infty$ and also that $\frac{1}{2^\alpha w(2)}\sum_{n=3}^\infty \frac{w(n)}{n^{(1-\alpha)}}\leq \frac{1}{2^\alpha}\sum_{n=3}^\infty w(n)<\infty$. So, again fix $m\in\N$ with $m\geq 3$ and  select $i\in\N$ with $2^i+1\leq m\leq 2^{i+1}$. As in Case (1),
\[
\sum_{n=m+1}^\infty \frac{w(n)}{n^{(1-\alpha)}}\leq \sum_{k=i}^\infty w(2^k+1)\sum_{n=2^k+1}^{2^{k+1}}\frac{1}{n^{1-\alpha}}.
\]
Since $(1-\alpha)>0$ and $\frac{1}{n^{1-\alpha}}\leq \frac{1}{(2^k+1)^{1-\alpha}}$ for $2^k+1\leq n\leq 2^{k+1}$ it follows that
\begin{eqnarray*}
& & \sum_{n=m+1}^\infty \frac{w(n)}{n^{(1-\alpha)}}\leq \sum_{k=i}^\infty w(2^k+1)\frac{2^k}{(2^k+1)^{1-\alpha}}\leq \sum_{k=i}^\infty w(2^k+1)\frac{2^k}{(2^k)^{1-\alpha}}\\
& & \ \ = \sum_{k=i}^\infty \frac{1}{2^k\cdot 2^{(k+1)2^{k+1}}}\cdot 2^{k\alpha}\leq \frac{1}{2^{(i+1)2^{i+1}}}\sum_{k=i}^\infty \left(\frac{1}{2^{1-\alpha}}\right)^k,
\end{eqnarray*}
where the last inequality uses the fact that $\frac{1}{2^{(k+1)2^{k+1}}}\leq \frac{1}{2^{(i+1)2^{i+1}}}$ for all $k\geq i$. Since, with $A:=1/(1-2^{\alpha-1})$, we have
\[
\frac{1}{2^{(i+1)2^{i+1}}}\sum_{k=i}^\infty \left(\frac{1}{2^{1-\alpha}}\right)^k=\frac{A}{2^{(i+1)2^{i+1}}}\cdot\frac{1}{2^{(1-\alpha)i}},
\]
we can conclude that
\begin{equation}\label{e.comnoo}
\sum_{n=m+1}^\infty \frac{w(n)}{n^{(1-\alpha)}}\leq\frac{A}{2^{(i+1)2^{i+1}}}\cdot\frac{1}{2^{(1-\alpha)i}}.
\end{equation}
Using $\frac{1}{w(m)}=2^i\cdot 2^{(i+1)2^{i+1}}$ and the inequality
\[
\frac{1}{m^\alpha}=m^{1-\alpha}\cdot\frac{1}{m}\leq (2^{i+1})^{1-\alpha}\cdot \frac{1}{(2^i+1)},
\]
it follows from \eqref{e.comnoo} and the inequality $\frac{2^i}{2^i+1}<1$ that
\[
\frac{1}{m^\alpha w(m)}\sum_{n=m+1}^\infty \frac{w(n)}{n^{(1-\alpha)}}\leq A\cdot \frac{2^i}{(2^i+1)}\cdot \frac{(2^{i+1})^{1-\alpha}}{2^{(1-\alpha)i}}\leq A\,2^{1-\alpha}<\infty.
\]
Again  \eqref{e.spe_car} in Theorem \ref{t:sp-dec-N}(ii) holds, i.e., $\lambda\in\rho(\sC^{(1,w)})$. The proof of Fact 4 and hence, the discussion of this example, is thereby complete.
\end{example}

To formulate the final result of this section we require some preliminaries. Let $w\in c_0$ be a decreasing, strictly positive  sequence. Then, with
a continuous inclusion, we have
\begin{equation}\label{e.inclusione}
\ell_p(w)\su c_0(w), \quad 1\leq p<\infty.
\end{equation}
For $p=1$ this is clear.
 Fix $1<p<\infty$. For $x\in \ell_p(w)$ we have
\[
|x_n|w(n)^{1/p}=(|x_n|^pw(n))^{1/p}\leq \left(\sum_{m=1}^\infty |x_m|^pw(m)\right)^{1/p}=\|x\|_{p,w}, \quad   n\in\N.
\]
Accordingly,
\[
0\leq |x_n|w(n)=|x_n|w(n)^{1/p}w(n)^{1/p'}\leq w(n)^{1/p'}\|x\|_{p,w}.
\]
Since $w\downarrow 0$, it follows that $\lim_{n\to\infty}|x_n|w(n)=0$, i.e., $x\in c_0(w)$ and
\[
\|x\|_{0,w}\leq \|w\|_\infty^{1/p'} \|x\|_{p,w},\quad x\in \ell_p(w).
\]
For the case   $w\in \ell_1$ with $w\downarrow 0$ we have, with a  continuous inclusion, that
\begin{equation}\label{e.in-p}
\ell_p(w)\su \ell_1(w),\quad 1<p<\infty.
\end{equation}
Indeed, define $\mu\colon 2^\N\to [0,\infty)$ by $\mu(A):=\sum_{n\in A}w(n)$, for $A\su \N$. Then $\mu$ is a \textit{finite}, positive measure  and  it is well known that $L^p(\mu)\su L^1(\mu)$, for $1<p<\infty$, with $\|f\|_1\leq \mu(\N)^{1/p'}\|f\|_p$. Accordingly,
\[
\|x\|_{1,w}\leq \left(\sum_{n=1}^\infty w(n)\right)^{1/p'}\|x\|_{p,w},\quad x\in \ell_p(w).
\]
The containment \eqref{e.in-p} does not always hold. Indeed,
let $w(n)=\frac{1}{\sqrt{n}}$, for $n\in\N$. Fix  $p\in (1,\infty)$. Then $x:=(\frac{1}{n^\alpha})_{n\in\N}$,   for any  fixed  $\alpha\in (\frac{1}{2p},\frac{1}{2}]$ satisfies  $x\in \ell_p(w)$ but $x\not \in \ell_1(w)$. Accordingly, $\ell_p(w)\not\su \ell_1(w)$ for all $1<p<\infty$.

\begin{prop}\label{C.1} Let $w\in c_0$ be  decreasing and  strictly positive. Then
\begin{equation}\label{e.inclspettro}
\cup_{1<p<\infty}\sigma_{pt}(\sC^{(p,w)})\su \sigma_{pt}(\sC^{(0,w)})\su \sigma_{pt}(\sC, \C^\N)=\Sigma.
\end{equation}
Suppose, in addition, that  $w\in \ell_1$ and $\sC^{(1,w)}\in \cL(\ell_1(w))$. Then
\begin{equation}\label{e.inclspettro1}
\cup_{1<p<\infty}\sigma_{pt}(\sC^{(p,w)})\su \sigma_{pt}(\sC^{(1,w)})\su\sigma_{pt}(\sC^{(0,w)})\su \Sigma.
\end{equation}
\end{prop}



The proof of the previous result is elementary and is therefore omitted.

\section{Iterates of $\sC^{(1,w)}$ and mean ergodicity}

For $X$ a Banach space, recall that  $T\in\cL(X)$ is   \textit{mean ergodic} (respectively, \textit{uniformly mean ergodic}) if its sequence of Ces\`aro averages
\begin{equation}\label{eq.cesaro-aver}
T_{[n]}:=\frac{1}{n}\sum_{m=1}^n T^m, \quad n\in\N,
\end{equation}
 converges to some operator $P\in \cL(X)$ in the strong operator topology $\tau_s$, i.e., $\lim_{n\to\infty}T_{[n]}x=Px$ for each $x\in X$, \cite[Ch.VIII]{DSI} (respectively, in the  operator norm  topology $\tau_b$). According to \cite[VIII Corollary 5.2]{DSI} there  then exists  the direct decomposition
\begin{equation}\label{eq.decompo}
X=\Ker (I-T)\oplus \ov{(I-T)(X)}.
\end{equation}
Moreover, we always have the identities
\begin{equation}\label{eq.ba}
    (I-T) T_{[n]} = T_{[n]}  (I-T) = \frac 1 n (T-T^{n+1}), \qquad n \in \N,
\end{equation}
and, setting  $T_{[0]} : = I$,  that
\begin{equation}\label{eq.bb}
    \frac 1 n T^n = T_{[n]} - \frac{(n-1)}{n} T_{[n-1]} , \qquad n \in \N .
\end{equation}

An operator $T\in \cL(X)$ is  \textit{Ces\`aro bounded} if   $\sup_{n\in\N}\|T_{[n]}\|<\infty$.
Every mean ergodic operator $T\in \cL(X)$ is necessarily Ces\`aro  bounded (by the Principle of Uniform Boundedness) and, via \eqref{eq.bb}, also satisfies
\begin{equation}\label{e.T_limite}
 \tau_s-\lim_{n\to\infty}\frac{1}{n}T^n=0.
\end{equation}
It is also clear from \eqref{eq.bb} that if $T$  is Ces\`aro bounded, then $\sup_{n\in\N}\frac{\|T^n\|}{n}<\infty$.
If $T\in \cL(X)$ is \textit{power bounded} (cf. Remark \ref{R.4n}(ii)), then  $T$ is also Ces\`aro  bounded and   $\lim_{n\to\infty}\frac{\|T^n\|}{n}=0$.
Condition  \eqref{e.T_limite} implies that  $\sigma(T)\su \ov{\mathbb{D}}$, \cite[p.709, Lemma 1]{DSI}, where $\mathbb{D}:=\{\lambda\in\C\colon |\lambda|<1\}$.

To characterize  the mean ergodicity of $\sC^{(1,w)}$ we require some preliminary facts.

\begin{lemma}\label{L.Mean-0} Let $w$ be   a bounded, strictly positive sequence such that  $\sC^{(1,w)}\in \cL(\ell_1(w))$. The following properties are satisfied.
\begin{itemize}
\item[\rm (i)] Each basis vector $e_r\in (I-\sC^{(1,w)})(\ell_1(w))$ for  $r\geq 2$.
\item[\rm (ii)] We have the equalities
\begin{equation}\label{eq.range-mean}
\ov{(I-\sC^{(1,w)})(\ell_1(w))}=\{x\in \ell_1(w)\colon x_1=0\}=\ov{{\rm span}\{e_r\colon r\geq 2\}}.
\end{equation}
\item[\rm (iii)]  The range of $I-\sC^{(1,w)}$  is closed  if and only if it coincides with
$$
\{x\in \ell_1(w)\colon x_1=0\}  .
$$
\item[\rm (iv)] The following three conditions are equivalent.
\begin{itemize}
\item[\rm (a)] $\Ker (I-\sC^{(1,w)})\not=\{0\}$.
\item[\rm (b)] $\Ker (I-\sC^{(1,w)})={\rm span}\{\mathbf{1}\}$.
\item[\rm (c)]  $\mathbf{1}\in \ell_1(w)$, that is, $w\in \ell_1$.
\end{itemize}
If $\mathbf{1}\not\in \ell_1(w)$, then  $\Ker (I-\sC^{(1,w)})=\{0\}$.
\end{itemize}
\end{lemma}

\begin{proof} (i) This follows from the identities
\[
e_{r+1}=(I-\sC^{(1,w)})(e_{r+1}-\frac{1}{r}\sum_{k=1}^re_k),\quad r\in\N,
\]
which can be verified by direct calculation.

(ii) Clearly,
$\{x\in \ell_1(w)\colon x_1=0\}=\ov{{\rm span}\{e_r\colon r\geq 2\}}$.
Part (i) implies
\[
\{x\in \ell_1(w)\colon x_1=0\}\su \ov{(I-\sC^{(1,w)})(\ell_1(w))}.
\]
On the other hand, since the $1$-st coordinate of $\sC^{(1,w)}x$ is $x_1$ for all $x\in \ell_1(w)$, we see that
\[
(I-\sC^{(1,w)})(\ell_1(w))\su\{x\in \ell_1(w)\colon x_1=0\}.
\]
The previous two  containments imply  \eqref{eq.range-mean}.

(iii) This is a direct consequence of part (ii) and the fact that the subspace  $\{x\in \ell_1(w)\colon x_1=0\}$  of $\ell_1(w)$  is closed.

(iv) The Ces\`aro operator $\sC\colon \C^\N\to \C^\N$ satisfies $\Ker (I-\sC)={\rm span}\{\mathbf{1}\}$.
Hence, $\Ker (I-\sC^{(1,w)})={\rm span}\{\mathbf{1}\}$ if and only if $\mathbf{1}\in \ell_1(w)$.

If $\mathbf{1}\not \in \ell_1(w)$, then  $ (I-\sC^{(1,w)})$ is injective, i.e., $\Ker (I-\sC^{(1,w)})=\{0\}$.
\end{proof}

\begin{lemma}\label{L-Mean1} Let   $w$ be   a bounded, strictly positive sequence such that  $\sC^{(1,w)}\in \cL(\ell_1(w))$. If $\sC^{(1,w)}$ is Ces\`aro bounded, then necessarily $w\in \ell_1$. In particular, this is the case whenever $\sC^{(1,w)}$ is power bounded or mean ergodic.
\end{lemma}

\begin{proof}  It is known that  $\sC\colon \C^\N\to\C^\N$ is power bounded, uniformly mean ergodic and satisfies both  $\Ker (I-\sC)={\rm span}\{\mathbf{1}\}$ and
\begin{equation}\label{eq.rango_C_N}
(I-\sC)(\C^\N)=\{x\in\C^\N\colon x_1=0\}=\ov{{\rm span}\{e_r\}_{r\geq 2}};
\end{equation}
see \cite[Proposition 4.1]{ABR-7}, \cite[Proposition 4.3]{ABR-1}.

Observe that the sequence $\{\sC_{[n]}e_1\}_{n\in\N}$ converges to $\mathbf{1}$ in $\C^\N$. Indeed, we have
 $e_1=\mathbf{1}-(0,1,1,1,\ldots)$ and, since $\sC\in \cL(\C^\N)$ is power bounded, that
\[
(I-\sC)(\C^\N)=\{x\in \C^N\colon \lim_{n\to\infty}\sC_{[n]}x=0\},
\]
\cite[Chap.VIII, \S 3, Theorem 1]{Y}.  Hence, the sequence
\[
\sC_{[n]}e_1=\sC_{[n]}\mathbf{1}-\sC_{[n]}(0,1,1,1,\ldots)=\mathbf{1}-\sC_{[n]}(0,1,1,1,\ldots),\quad n\in\N,
\]
 converges to $\mathbf{1}$ in $\C^\N$ as $n\to\infty$
because
 $(0,1,1,\ldots)\in (I-\sC)(\C^\N)$ by \eqref{eq.rango_C_N}.

We now proceed to verify that $w\in \ell_1$.
By assumption $\sC^{(1,w)}$ is Ces\`aro bounded and so  $\{\sC^{(1,w)}_{[n]}e_1\}_{n\in\N}$ is a bounded subset of $\ell_1(w)$. By Alaoglu's theorem
all norm closed  balls of $\ell_1(w)$ are $\sigma(\ell_1(w),c_0(w^{-1}))$-compact (i.e., weakly$^*$ compact) and, equipped with the topology $\sigma(\ell_1(w),c_0(w^{-1}))$, they are  metrizable because $c_0(w^{-1})$ is a separable Banach space, \cite[Corollary 2.6.20]{ME}.
 Therefore, there is a subsequence $\{\sC^{(1,w)}_{[n(k)]}e_1\}_{k\in\N}$ of $\{\sC^{(1,w)}_{[n]}e_1\}_{n\in\N}$ and a vector $u\in \ell_1(w)$ such that $\sC^{(1,w)}_{[n(k)]}e_1\to u$ for the topology  $\sigma(\ell_1(w),c_0(w^{-1}))$ as $k\to\infty$. Since the topology $\sigma(\ell_1(w),c_0(w^{-1}))$ is finer than the topology of coordinatewise convergence in $\ell_1(w)$, we can conclude  that  $\sC^{(1,w)}_{[n(k)]}e_1=\sC_{[n(k)]}e_1\to u$ in $\C^\N$ as $k\to\infty$.
The previous paragraph then implies that  $u=\mathbf{1}$ and so $\mathbf{1}\in \ell_1(w)$. In other words, $w\in \ell_1$.
\end{proof}

\begin{remark}\label{R.43} \rm
If $0 < \alpha \le 1 $, then the weight $w_\alpha := \left(\frac{1}{n^\alpha} \right)_{n \in \N}$ satisfies $w_\alpha \notin \ell_1$. By Lemma \ref{L-Mean1},
$\sC^{(1,w_\alpha )}$ is not Cesàro bounded. The same is true for the weight $w$ in Remark \ref{R.VW}(ii).
\end{remark}

\begin{lemma}\label{L-Mean2} Let   $w$ be   a bounded, strictly positive sequence such that $w\in \ell_1$ and  $\sC^{(1,w)}\in \cL(\ell_1(w))$. Then
\begin{equation}\label{eq.decomp_1}
\ell_1(w)=\Ker (I-\sC^{(1,w)})\oplus \ov{(I-\sC^{(1,w)})(\ell_1(w))}.
\end{equation}
\end{lemma}

\begin{proof}
Set $f_1:=\mathbf{1}$ and  define $f_j:=f_1-\sum_{k=1}^{j-1}e_k$ for $j\geq 2$. Since  $w\in \ell_1$, we have $\{f_j\}_{j\in\N}\su \ell_1(w)$.
Moreover, \eqref{eq.range-mean} reveals that
\[
\{f_j\}_{j\geq 2}\su \{x\in \ell_1(w)\colon x_1=0\}=\ov{{\rm span}\{e_r\colon r\geq 2\}}.
\]
In particular, this implies that
\[
e_1=(f_1-f_2)\in {\rm span}\{f_1\}\oplus \ov{{\rm span}\{e_r\colon r\geq 2\}}.
\]
Since  $\{e_r\}_{r\in\N}$ is a basis for $\ell_1(w)$, it follows that
\[
\ell_1(w)={\rm span}\{f_1\}\oplus \ov{{\rm span}\{e_r\colon r\geq 2\}}.
\]
The conclusion now follows from Lemma \ref{L.Mean-0}(ii), (iv).
\end{proof}

Let  $m\in\N$. According to  \cite[Sect. 11.12]{H}, $\sC^m$ is the moment difference operator for the measure on  $[0,1]$ given by $d\mu=f_m(t)\,dt$, with
\[
f_m(t):=\frac{1}{(m-1)!}\log^{m-1}\left(\frac{1}{t}\right),\quad  t\in (0,1].
\]
Therefore, the   identities
\begin{equation}\label{eq.iter}
(\sC^mx)_n=\sum_{k=1}^n\left(\begin{array}{c} n-1\\ k-1\end{array}\right)x_k\int_0^1t^{k-1}(1-t)^{n-k}f_m(t)\,dt,\quad n\in\N,
\end{equation}
hold for all $x\in \C^\N$;
see also  \cite[p.125]{L}.

\begin{lemma}\label{L-Mean3} Let   $w$ be   a bounded, strictly positive sequence such that $w\in \ell_1$ and  $\sC^{(1,w)}\in \cL(\ell_1(w))$. Then, for each  $r\geq 2$, the sequence  $\{(\sC^{(1,w)})^m e_r\}_{m\in\N}$ converges to $0$ in $\ell_1(w)$.
\end{lemma}

\begin{proof} Fix  $r\geq 2$.  By  \eqref{eq.iter}, for each $m\in\N$, we have  $((\sC^{(1,w)})^me_r)_n=0$ for $1\leq n<r$ and
\begin{equation}\label{eq.unit}
((\sC^{(1,w)})^me_r)_n=\left(\begin{array}{c} n-1\\ r-1\end{array}\right)\int_0^1t^{r-1}(1-t)^{n-r}f_m(t)\,dt, \quad n\geq r.
\end{equation}
Proceeding as in the proof of \cite[Theorem 1]{GF}, define $g_m(0):=0$, $g_m(t):=tf_m(t)$ for $0<t\leq 1$ and
\[
a_m:=\sup\{g_m(t)\colon t\in [0,1]\}, \quad m\in\N.
\]
For each  $m\in\N$ we obtain  that
$|((\sC^{(1,w)})^me_r)_n|\leq \frac{a_m}{r-1}$ for all $n\in\N$. Hence,
\[
w(n)|((\sC^{(1,w)})^me_r)_n|\leq \frac{w(n)a_m}{r-1},\quad  n\in\N,
\]
from which  it follows that
\[
\|(\sC^{(1,w)})^me_r\|_{1,w}\leq \|w\|_1\frac{a_m}{r-1}, \quad  m\in\N.
\]
According to  \cite[Lemma 1]{GF} we have
$\lim_{m\to\infty}a_m= 0$, 	
which implies the desired conclusion.
\end{proof}

We can now establish the first main result of this section.

\begin{theorem}\label{T-conv} Let   $w$ be   a bounded, strictly positive sequence such that $\sC^{(1,w)}\in \cL(\ell_1(w))$.
\begin{itemize}
\item[\rm (i)] $\sC^{(1,w)}$ is power bounded if and only if $\{(\sC^{(1,w)})^m\}_{m\in\N}$ converges in $\cL_s(\ell_1(w))$ to the projection onto $\Ker (I-\sC^{(1,w)})$ along $\ov{(I-\sC^{(1,w)})(\ell_1(w))}$.

In this case, $\sC^{(1,w)}$ is necessarily mean ergodic.
\item[\rm (ii)] $\sC^{(1,w)}$ is mean ergodic if and only if $\sC^{(1,w)}$ is Ces\`aro  bounded.
\end{itemize}
\end{theorem}

\begin{proof} (i) Assume  that $\sC^{(1,w)}$ is power bounded. Then $w\in \ell_1$ by Lemma \ref{L-Mean1}.
It  follows from  Lemma \ref{L-Mean2} that \eqref{eq.decomp_1} holds and
from Lemma \ref{L.Mean-0} that
$\Ker (I-\sC^{(1,w)})={\rm span}\{\mathbf{1}\}$ and $\ov{(I-\sC^{(1,w)})(\ell_1(w))}=\ov{{\rm span}}\{e_r\}_{r\geq 2}$.
So, by \eqref{eq.decomp_1}, for  $x\in \ell_1(w)$
 we have  $x=y+z$ with $y\in \Ker (I-\sC^{(1,w)})$ and $z\in \ov{(I-\sC^{(1,w)})(\ell_1(w))}$. Then, for each $m\in\N$, it follows that
\begin{equation}\label{eq.decomvector}
(\sC^{(1,w)})^m x=(\sC^{(1,w)})^m y+(\sC^{(1,w)})^mz=y+(\sC^{(1,w)})^mz.
\end{equation}
Moreover, for each $r\geq 2$,   $\lim_{m\to\infty}(\sC^{(1,w)})^me_r= 0$ in $\ell_1(w)$; see Lemma \ref{L-Mean3}.
  Since $\sC^{(1,w)}$ is power bounded and ${\rm span}\{e_r\}_{r\geq 2}$ is dense in $\ov{(I-\sC^{(1,w)})(\ell_1(w))}$ (cf. Lemma \ref{L.Mean-0}(ii)), it follows that $\lim_{m\to\infty}(\sC^{(1,w)})^mz= 0$ in $\ell_1(w)$  for each $z\in \ov{(I-\sC^{(1,w)})(\ell_1(w)))}$. Hence,  $\lim_{m\to\infty}(\sC^{(1,w)})^m x=y$ in $\ell_1(w)$; see \eqref{eq.decomvector}.

The assumption of the reverse implication  implies, in particular, that $\{(\sC^{(1,w)})^m\}_{m\in\N}$ converges in $\cL_s(\ell_1(w))$ and so, by the Principle of Uniform Boundedness, $\sC^{(1,w)}$ is power bounded.

If a sequence in a locally convex Hausdorff space (briefly, lcHs) is convergent,  so is its sequence of averages (to the same limit). Hence, the convergence  of $\{(\sC^{(1,w)})^m\}_{m\in\N}$  in the lcHs $\cL_s(\ell_1(w)):=(\cL(\ell_1(w)),\tau_s)$ implies the convergence  of $\{\sC^{(1,w)}_{[n]}\}_{n\in\N}$ in $\cL_s(\ell_1(w))$, i.e., $\sC^{(1,w)}$ is mean ergodic.

(ii) If $\sC^{(1,w)}$ is mean ergodic then, as noted before, $\sC^{(1,w)}$ is also Ces\`aro  bounded.

Assume now that $\sC^{(1,w)}$ is Ces\`aro  bounded, in which case $w\in \ell_1$ (cf. Lemma \ref{L-Mean1}).
Again by Lemma \ref{L-Mean2} we see that \eqref{eq.decomp_1} holds. We need to verify that $\{\sC^{(1,w)}_{[n]}\}_{n\in\N}$ is a convergent sequence in $\cL_s(\ell_1(w))$.
This follows from an argument similar to the one in part (i).
\end{proof}

The following result should be compared with Lemma \ref{L-Mean3}.

\begin{corollary}\label{C.WN} Let $w$ be a bounded, strictly positive sequence such that $\sC^{(1,w)}\in \cL(\ell_1(w))$   is power bounded. Then $\lim_{m\to\infty} (\sC^{(1,w)})^me_1=\mathbf{1}$.
\end{corollary}

\begin{proof} Lemma \ref{L-Mean1} implies that $\mathbf{1}\in \ell_1(w)$. According to Theorem \ref{T-conv}(i) there exists $u\in \Ker (I-\sC^{(1,w)})$ such that
\begin{equation}\label{eq.limite_1}
\lim_{m\to\infty} (\sC^{(1,w)})^me_1=u \ \mbox{ in }\ \ell_1(w).
\end{equation}
Since $\Ker (I-\sC^{(1,w)})={\rm span}\{\mathbf{1}\}$, there exists $\lambda\in \C$ such that $u=\lambda \mathbf{1}$. But, the $1$-st coordinate of  $(\sC^{(1,w)})^me_1$ equals $1$ for every $m\in\N$ and so it follows from \eqref{eq.limite_1} that $\lambda=1$.
\end{proof}

\begin{remark}\label{R.NSW}\rm Theorem \ref{T-conv} is  special for the Ces\`aro operator acting in $\ell_1(w)$ and is not valid for a general Banach space operator $T\in \cL(X)$.

Indeed, concerning part (i) of Theorem \ref{T-conv}, the proof shows that whenever $\{T^m\}_{m\in\N}$ converges in $\cL_s(X)$, then $T$ is necessarily power bounded. To see that the converse is false in general, consider the Banach space $X=C([0,1])$ equipped with the sup-norm $\|\cdot\|_\infty$ and define $T\in \cL(X)$ by $Tf:=\varphi f$ for $f\in X$, where $\varphi(t):=t$ for $t\in [0,1]$. Since $T^mf=\varphi^m f$ for $f\in X$, with $\|\varphi^m\|_\infty\leq 1$ for all $m\in\N$, it is clear that $T$ is power bounded. However, if ${\mathbf 1}$ is the function constantly equal to $1$ in $[0,1]$, then the sequence $T^m{\mathbf 1}=\varphi^m$, $m\in\N$, converges pointwise on $[0,1]$ to the discontinuous function $\chi_{\{1\}}$. In particular, $\{T^m{\mathbf 1}\}_{m\in\N}$ cannot be a convergent sequence in $X$.

Concerning part (ii) of Theorem \ref{T-conv}, the mean ergodicity of an operator always implies its Ces\`aro boundedness. To see that the converse is false in general,
let $X$ and $T$ be as in the previous paragraph. Since $T$ is power bounded, it is  Ces\`aro bounded. But, $T$ is not mean ergodic. Indeed,
\[
T_{[n]}{\mathbf 1}=\frac{1}{n}\sum_{m=1}^n\varphi^m,\quad n\in\N.
\]
Since $\left(\frac{1}{n}\sum_{m=1}^n\varphi^m\right)(t)=\frac{t-t^{n+1}}{n(1-t)}$ for $t\in [0,1)$ and $\left(\frac{1}{n}\sum_{m=1}^n\varphi^m\right)(1)=1$, for all $n\in\N$, it is clear that $\{T_{[n]}{\mathbf 1}\}_{n\in\N}$ converges pointwise on $[0,1]$ to the discontinuous function $\chi_{\{1\}}$. In particular, $\{T_{[n]}{\mathbf 1}\}_{n\in\N}$ cannot be a convergent sequence in $X$ and so $T$ is not mean ergodic.
\end{remark}


Given a    bounded, strictly positive sequence $w$, for the remainder of this section we use the notation
\[
X_1(w):=\{x\in \ell_1(w)\colon x_1=0\},
\]
which is always a closed subspace of $\ell_1(w)$. In the event that  $\sC^{(1,w)}\in \cL(\ell_1(w))$, the  subspace $X_1(w)$ is clearly  invariant for $\sC^{(1,w)}$; see \eqref{eq.Ce-op}.

\begin{lemma}\label{L.range} Let $w$ be   a bounded, strictly positive sequence such that  $w\in \ell_1$ and  $\sC^{(1,w)}\in \cL(\ell_1(w))$.  Then
\begin{equation}\label{eq.range_0}
(I-\sC^{(1,w)})(\ell_1(w))=(I-\sC^{(1,w)})(X_1(w)).
\end{equation}
\end{lemma}

\begin{proof} Clearly, $(I-\sC^{(1,w)})(X_1(w))\su (I-\sC^{(1,w)})(\ell_1(w))$.

To verify  the reverse inclusion, we proceed as in the proof of  \cite[Lemma 4.5]{ABR-9}. First observe, via \eqref{eq.Ce-op},  that for each
$x\in \ell_1(w)$ we have
\begin{equation}\label{eq.rango1}
(I-\sC^{(1,w)})x=\left(0, x_2-\frac{x_1+x_2}{2},x_3-\frac{x_1+x_2+x_3}{3},\ldots\right),
\end{equation}
and, in particular,  for each $ y\in X_1(w) $  that
\begin{equation}\label{eq.rango2}
(I-\sC^{(1,w)})y=\left(0, \frac{y_2}{2},y_3-\frac{y_2+y_3}{3},y_4-\frac{y_2+y_3+y_4}{4},\ldots\right) .
\end{equation}
Fix $x\in \ell_1(w)$. We apply  \eqref{eq.rango1} to conclude that
\begin{equation}\label{eq.rango3}
x_j-\frac{1}{j}\sum_{k=1}^jx_k=\frac{1}{j}\left((j-1)x_j-\sum_{k=1}^{j-1}x_k\right),\quad j\geq 2,
\end{equation}
is the $j$-th coordinate of the vector $(I-\sC^{(1,w)})x$. Set $y_i:=x_i-x_1$ for all $i\in\N$.  Then the vector $y:=(y_i)_{i\in\N}$
belongs to $ X_1(w)$ because  $w\in \ell_1$  implies that $(0,1,1,1,\ldots)\in \ell_1(w)$.  We  apply  \eqref{eq.rango2} to conclude
that the $j$-th coordinate of $(I-\sC^{(1,w)})y$ is given by \eqref{eq.rango3} for $j\geq 2$. Hence,
\[
(I-\sC^{(1,w)})x=(I-\sC^{(1,w)})y\in (I-\sC^{(1,w)})(X_1(w)).
\]
\end{proof}

\begin{remark}\label{R.WW_1}\rm The equality \eqref{eq.range_0} \textit{fails} whenever $w\not\in \ell_1 $ and $\sC^{(1,w)}\in \cL(\ell_1(w))$. Indeed, in this case Lemma \ref{L.Mean-0}(iv) implies that $(I-\sC^{(1,w)})$ is injective. This implies that $x:=(I-\sC^{(1,w)})e_1$ cannot belong to $(I-\sC^{(1,w)})(X_1(w))$. So, the containment
\[
(I-\sC^{(1,w)})(X_1(w))\subsetneqq (I-\sC^{(1,w)})(\ell_1(w))
\]
is proper whenever $w\not\in \ell_1$. For the existence of weights $w\not\in \ell_1$ such that $\sC^{(1,w)}\in \cL(\ell_1(w))$ see Remark \ref{R.VW}(ii) and also Examples \ref{esempi-C}(ii) with $0<\alpha\leq 1$.
\end{remark}

Given a bounded, strictly  positive weight $w=(w(n))_{n\in\N}$, we introduce the associated quantity
\begin{equation}\label{eq.asso}
\cU_w:=\sup_{m\in\N}\frac{1}{mw(m+1)}\sum_{n=m+1}^\infty w(n)=\sup_{r\geq 2}\frac{1}{(r-1)w(r)}\sum_{n=r}^\infty w(n).
\end{equation}
It turns out that $\cU_w$ is useful for determining certain mean ergodic and related properties of $\sC^{(1,w)}$. As a sample, it is clear that $\cU_w<\infty$ implies $w\in \ell_1$. Moreover, $\cU_w<\infty$ also implies that $\sC^{(1,w)}\in \cL(\ell_1(w))$. This follows directly from Proposition \ref{P1-C}(i) and the inequality
\[
\frac{1}{w(m+1)}\sum_{n=m+1}^\infty\frac{w(n)}{n}\leq \frac{1}{mw(m+1)}\sum_{n=m+1}^\infty w(n),\quad m\in\N.
\]

The following result characterizes the condition $\cU_w<\infty$.

\begin{prop}\label{L.range_closed} Let $w$ be   a bounded, strictly positive sequence such that $w\in \ell_1$ and  $\sC^{(1,w)}\in \cL(\ell_1(w))$. The following  conditions are equivalent.
\begin{itemize}
\item[\rm (i)] The range of $I-\sC^{(1,w)}  $ is   closed in $\ell_1(w)$.
\item[\rm (ii)] $(I-\sC^{(1,w)})(\ell_1(w))=X_1(w)$.
\item[\rm (iii)] $(I-\sC^{(1,w)})(X_1(w))=X_1(w)$.
\item[\rm (iv)] The quantity $\cU_w<\infty$.
\end{itemize}
\end{prop}

\begin{proof} (i)$\Leftrightarrow$(ii) follows from Lemma \ref{L.Mean-0}(iii) and the definition of $X_1(w)$.

(ii)$\Leftrightarrow$(iii) is clear from Lemma \ref{L.range}.

(iii)$\Leftrightarrow$(iv). First observe (via \eqref{eq.Ce-op}) that $(I-\sC^{(1,w)})$ maps $X_1(w)$ into itself, and that the restriction $(I-\sC^{(1,w)})\colon X_1(w)\to X_1(w)$ is both continuous and injective. The injectivity follows from Lemma \ref{L.Mean-0}(iv) as $\mathbf{1}\not\in X_1(w)$.

According to the previous paragraph, condition (iii) is equivalent to the restricted operator  $(I-\sC^{(1,w)})\colon X_1(w)\to X_1(w)$ being bijective (i.e.,  surjective). By the Open Mapping Theorem this, in turn, is equivalent to $(I-\sC^{(1,w)})\colon X_1(w)\to X_1(w)$ having a continuous inverse.

So, (iii)$\Leftrightarrow$(iv) is equivalent to showing that (iv) holds if and only if the operator $(I-\sC^{(1,w)})\colon X_1(w)\to X_1(w)$ is bijective with a continuous inverse.
To do this we first note, with  $\tilde{w}(n):=w(n+1)$ for $ n\in\N$, that the linear  shift operator $S\colon X_1(w)\to \ell_1(\tilde{w})$ defined by
\[
S(x):=(x_2,x_3,\ldots), \quad x\in X_1(w),
\]
is an isometric isomorphism 	
	of $X_1(w)$  onto $\ell_1(\tilde{w})$. So, it suffices to verify
\[
A:=S\circ (I-\sC^{(1,w)})|_{X_1(w)}\circ S^{-1}\in \cL(\ell_1(\tilde{w})),
\]
which is given by the formula
\begin{equation}\label{eq.formulaW}
Ax=\left(\frac{1}{n+1}\left(nx_n-\sum_{k=1}^{n-1}x_k\right)\right)_{n\in\N}, \quad x\in \ell_1(\tilde{w}),
\end{equation}
with $x_0:=0$ (see the purely algebraic calculations in the proof of Lemma 4.5 in \cite{ABR-9}),
is bijective with a continuous inverse if and only if (iv) holds.

Now the operator $A$ given by \eqref{eq.formulaW}, when considered from  $\C^\N$ to $\C^\N$,  is bijective  and routine calculations show that its inverse map $B\colon \C^\N\to \C^\N$ is determined by the lower triangular matrix $B=(b_{nm})_{n,m\in\N}$ with entries given by
  $b_{nm}=0$ if $m>n$, $b_{nm}=\frac{n+1}{n}$ if $m=n$ and $b_{nm}=\frac{1}{m}$ if $1\leq m<n$. The restriction of the linear map $B$ acts continuously from $\ell_1(\tilde{w})$ into itself if and only if
$D:=\Phi_{\tilde{w}}\circ B\circ \Phi_{\tilde{w}}^{-1}$ belongs to $\cL(\ell_1)$, where   $\Phi_{\tilde{w}}\colon \ell_1(\tilde{w})\to \ell_1$ is the  isometric isomorphism    given by
\[
 \Phi_{\tilde{w}}(x):=(w(n+1)x_n)_{n\in\N},\quad x\in \ell_1(\tilde{w}).
\]
Of course, both $\Phi_{\tilde{w}}$ and $\Phi_{\tilde{w}}^{-1}$ can be extended to isomorphisms between $\C^\N$ (which we denote by the same symbol as no confusion can occur). The linear operator $D$ (considered from  $\C^\N$ into itself) is associated with the lower triangular matrix
 $(\frac{w(n+1)}{w(m+1)}b_{nm})_{n,m\in\N}$. By Lemma \ref{L_U}, $D\in \cL(\ell_1)$ if and only if $\sup_{m\in\N}\sum_{n=1}^\infty\frac{w(n+1)b_{nm}}{w(m+1)}<\infty$. Since $w\in \ell_1$ and $\lim_{m\to\infty}\frac{m+1}{m}=1$, this condition is equivalent to $\cU_w<\infty$ (see \eqref{eq.asso}). This completes the proof of (iii)$\Leftrightarrow$(iv).
\end{proof}

\begin{prop}\label{P.SC-Bound} Let   $w$ be   a bounded, strictly positive sequence such that $\cU_w<\infty$. Then  $\sC^{(1,w)}$ is  power bounded and uniformly mean ergodic.
\end{prop}

\begin{proof} It was already noted that $\cU_w<\infty$ implies  $w\in \ell_1$ and $\sC^{(1,w)}\in\cL(\ell_1(w))$.  From the proof of  Lemma \ref{L-Mean3} (and its notation) recall that
\begin{equation}\label{eq.re}
|(\sC^{(1,w)})^me_r)_n|\leq \frac{a_m}{r-1},\quad n\in\N,
\end{equation}
for $m\in\N$ and $r\geq 2$. Moreover, $a_m\geq\frac{1}{2}f_m(\frac{1}{2})>0$ and $\lim_{m\to\infty}a_m=0$.

Since $(\sC^{(1,w)})^m\mathbf{1}=\mathbf{1}$ for all $m\in\N$ and $\ell_1(w)={\rm span}\{\mathbf{1}\}\oplus X_1(w)$ (by Lemma \ref{L.Mean-0} and Lemma \ref{L-Mean2}), to show that $\sC^{(1,w)}$ is  power bounded it suffices to show that $\sup_{m\in\N}\|(\sC^{(1,w)})^mx\|_{1,w}<\infty$ for each $x=(0,x_2,x_3,\ldots)\in X_1(w)$. So,
 fix such an  $x\in X_1(w)$, in which case  $x=\sum_{r=2}^\infty x_re_r$. Recall from the proof Lemma \ref{L-Mean3}, for each $m\in\N$, that $(\sC^{(1,w)})^me_r)_n=0$ if $r\geq 2$ and $1\leq n<r$. Accordingly, for each $n,\ m\in\N$ we have
\begin{eqnarray*}
w(n)|((\sC^{(1,w)})^mx)_n|&\leq & w(n)\sum_{r=2}^\infty  |x_r|\cdot |((\sC^{(1,w)})^me_r)_n|\\
&= & w(n)\sum_{r=2}^n |x_r|\cdot |((\sC^{(1,w)})^me_r)_n|.
\end{eqnarray*}
Hence, for every $m,\ n\in\N$ and $x\in X_1(w)$ it follows that
\begin{eqnarray*}
\|(\sC^{(1,w)})^mx\|_{1,w}&=&\sum_{n=2}^\infty w(n)|((\sC^{(1,w)})^mx)_n|\\
&\leq & \sum_{n=2}^\infty w(n)\sum_{r=2}^n |x_r|\cdot|((\sC^{(1,w)})^me_r)_n|\\
&=&\sum_{r=2}^\infty w(r) |x_r|\frac{1}{w(r)}\sum_{n=r}^\infty w(n)|((\sC^{(1,w)})^me_r)_n|\\
&\leq & \|(a_m)_{m\in\N}\|_\infty\|x\|_{1,w}\sup_{r\geq 2}\frac{1}{(r-1) w(r)}\sum_{n=r}^\infty w(n),
\end{eqnarray*}
where the last inequality relies on \eqref{eq.re}. An examination of \eqref{eq.asso} now shows that $\cU_w<\infty$ implies that $\sup_{m\in\N}\|(\sC^{(1,w)})^mx\|_{1,w}<\infty$ for each $x\in X_1(w)$. As already noted, this yields that $\sC^{(1,w)}$ is  power bounded.

Using now the fact that $\sC^{(1,w)}\in \cL(\ell_1(w))$ is  power bounded, we have $\lim_{m\to\infty} \frac{\|\sC^{(1,w)})^m\|}{m}=0$. Since also the range of $I-\sC^{(1,w)}$ is a closed subspace of $\ell_1(w)$ (cf. Proposition \ref{L.range_closed}), we can apply a result of Lin, \cite[Theorem]{Li}, to conclude that $\sC^{(1,w)}$ is  uniformly mean ergodic.
\end{proof}

\begin{remark}\label{R.UM1}\rm  Let $w$ be a bounded, strictly positive sequence such that $\sC^{(1,w)}\in \cL(\ell_1(w))$ and $\lim_{m\to\infty} \frac{\|(\sC^{(1,w)})^m\|}{m}=0$. Then $\sC^{(1,w)}$ is uniformly mean ergodic if and only if  $\cU_w<\infty$. Indeed, by Proposition \ref{P.SC-Bound} the condition $\cU_w<\infty$ implies uniform mean ergodicity. On the other hand, if $\sC^{(1,w)}$ is uniformly mean ergodic (in which case $w\in\ell_1$ by Lemma \ref{L-Mean1}), then Lin's theorem, \cite{Li}, ensures that $I-\sC^{(1,w)}$ has closed range in $\ell_1(w)$. Hence, $\cU_w<\infty$; see Proposition \ref{L.range_closed}.
\end{remark}

\begin{prop}\label{P.SC-Bound-C} Let $w$ be a bounded, strictly positive sequence such that $\sC^{(1,w)}\in \cK(\ell_1(w))$. Then necessarily $\cU_w<\infty$.

In particular, $\sC^{(1,w)}$ is both power bounded and uniformly mean ergodic.
\end{prop}

\begin{proof} Proposition \ref{P_SR}(ii) shows $w\in\ell_1$. Moreover,  the compactness of $\sC^{(1,w)}$  implies that $(I-\sC^{(1,w)})(\ell_1(w))$ is
 closed in $\ell_1(w)$, \cite[Lemma 3.4.20]{ME}.  Now apply Proposition \ref{L.range_closed} to conclude that $\cU_w<\infty$.
 Hence, $\sC^{(1,w)}$ is  power bounded and uniformly mean ergodic by Proposition \ref{P.SC-Bound}.
\end{proof}

\begin{remark}\label{R.Limite-Compact}\rm According to Proposition \ref{P2-C}, $\sC^{(1,w)}$ is compact whenever  $\limsup_{n\to\infty}\frac{w(n+1)}{w(n)}\in [0,1)$. In particular, this is the case for $w=(n^\beta r^n)_{n\in\N}$ with $r\in (0,1)$ and $\beta\geq 0$, for $w=(\frac{1}{n^n})_{n\in\N}$ and for $w=(\frac{a^n}{n!})_{n\in\N}$ with $a>0$; see Examples \ref{E.COP}(i)-(iii). Proposition \ref{P.SC-Bound-C} implies in all cases that $\sC^{(1,w)}$ is  power bounded and uniformly mean ergodic. By the same reasoning the  Ces\`aro operator corresponding to each of the weights in (iv), (v), (vi) of Example \ref{E.COP} is power bounded and uniformly mean ergodic.
\end{remark}

\begin{example}\label{R.L-Compact}\rm (i) Consider $w_\alpha(n)=(\frac{1}{n^\alpha})_{n\in\N}$ for fixed $\alpha>0$.
 For $\alpha\in (0,1]$,   Remark \ref{R.43} implies that  $\sC^{(1,w_\alpha)}\in \cL(\ell_1(w_\alpha))$ is \textit{not} Ces\`aro bounded and hence, is \textit{neither} mean ergodic nor power bounded.
The same is true for the weight $w$ in Remark \ref{R.VW}(ii).

On the other hand if $\alpha>1$, then it follows from Lemma \ref{L-stiam2} that $\sum_{n=m}^\infty\frac{1}{n^\alpha} \leq \frac{1}{(\alpha-1)(m-1)^{\alpha-1}}$. Thus, for each $m\geq 2$,
\[
\frac{1}{(m-1)w_\alpha(m)}\sum_{n=m}^\infty w_\alpha(n)\leq \frac{m^\alpha}{(\alpha-1)(m-1)^\alpha}\leq \frac{2^\alpha}{(\alpha-1)}
\]
and so $\cU_{w_\alpha}<\infty$.
Hence,   Proposition \ref{P.SC-Bound} implies that  $\sC^{(1,w_\alpha)}$ is power bounded and uniformly mean ergodic. However, $\sC^{(1,w_\alpha)}$ is \textit{not} compact; see Remark \ref{R.W_N}(iii). Observe that $w_\alpha\in \ell_1\setminus s$ for $ \alpha >1$.

(ii) Let $w\in s$ be the weight considered in Example \ref{ex.nocompact}. The claim is that $\cU_w<\infty$. To see this
fix $n\in\N$ with $n\geq 2$ and choose $i\in\N$ such that $2^i+1\leq n\leq 2^{i+1}$. Then
\[
\sum_{m=n}^\infty w(m)\leq \sum_{m=2^i+1}^\infty w(m)=\sum_{j=i}^\infty\sum_{m=2^j+1}^{2^{j+1}}w(m).
\]
Since each sum $\sum_{m=2^j+1}^{2^{j+1}}(\ldots )$ has $2^j$ terms and $w(m)=\frac{1}{2^j2^{(j+1)2^{j+1}}}$ for all $2^j+1\leq m\leq 2^{j+1}$, it follows that
\[
\sum_{m=n}^\infty w(m)\leq\sum_{j=i}^\infty 2^j\frac{1}{2^j2^{(j+1)2^{j+1}}}=\sum_{j=i}^\infty \frac{1}{2^{(j+1)2^{j+1}}}.
\]
As $\frac{1}{(n-1)}\leq \frac{1}{2^i}$ and $\frac{1}{w(n)}=2^i2^{(i+1)2^{i+1}}$, the previous inequality implies
\begin{eqnarray*}
 \frac{1}{(n-1)w(n)}\sum_{m=n}^\infty w(m)& \leq & \frac{1}{2^i}\cdot 2^i2^{(i+1)2^{i+1}}\sum_{j=i}^\infty \frac{1}{2^{(j+1)2^{j+1}}}\\
& = & 2^{(i+1)2^{i+1}}\sum_{j=i}^\infty\frac{1}{2^{j+1}}\cdot \frac{1}{2^{(j+1)(2^{j+1}-1)}}.
\end{eqnarray*}
But, $\frac{1}{2^{(j+1)(2^{j+1}-1)}}\leq \frac{1}{2^{(i+1)(2^{i+1}-1)}}$ for all $j\geq i$ and so
\[
\frac{1}{(n-1)w(n)}\sum_{m=n}^\infty w(m)\leq 2^{(i+1)2^{i+1}}\cdot \frac{1}{2^{(i+1)(2^{i+1}-1)}}\sum_{j=i}^\infty\frac{1}{2^{j+1}}=2.
\]
According to \eqref{eq.asso} we have $\cU_w\leq 2$. Then Proposition \ref{P.SC-Bound} shows that $\sC^{(1,w)}$ is power bounded and uniformly mean ergodic. But, $\sC^{(1,w)}$ is \textit{not} compact; see Fact 3 in Example \ref{ex.nocompact}.
\end{example}

The final example  exhibits  features different to the previous examples (eg. $\cU_w=\infty$). Its spectrum is also precisely determined.

\begin{example}\label{E.an}\rm  Let $\alpha>1$. Define the bounded, strictly positive weight $w$ by $w(1)=w(2):=1$ and $w(n):=\frac{1}{i^\alpha 2^{i-1}}$ for $2^{i}+1\leq n\leq  2^{i+1}$ and $i\in\N$. We record various properties of $w$.

\textit{Fact 1. $w\in \ell_1$, but $w\not\in s$}.

Define $v:=\left(\frac{1}{n\log^\alpha (n+1)}\right)_{n\in\N}$. It is shown in \eqref{eq.comp_1} of Remark \ref{R.3n}(iv) that $A_1 v\leq w \leq A_2v$ for positive constants $A_1, \ A_2$. The integral test for convergence of series implies that $v\in \ell_1$ and hence, also $w\in \ell_1$. Clearly, $v\not \in s$ and so also $w\not\in s$.

%

\textit{Fact 2. $\sC^{(1,w)}\in\cL(\ell_1(w))$.}

This was established in Remark \ref{R.3n}(iv).


\textit{Fact 3. $\cU_w=\infty$.}

Fix  $m\geq 3$ and choose $i\in\N$ to satisfy $2^i+1\leq m\leq 2^{i+1}$. Then
\begin{eqnarray*}
& & \frac{1}{(m-1) w(m)}\sum_{n=m}^\infty w(n)\geq   \frac{1}{(m-1) w(m)}\sum_{n=2^{i+1}+1}^\infty  w(n)\\
& & =\frac{1}{(m-1) w(m)}\sum_{j=i+1}^\infty\sum_{n=2^j+1}^{2^{j+1}}\frac{1}{j^\alpha 2^{j-1}}
= \frac{1}{(m-1) w(m)}\sum_{j=i+1}^\infty\frac{2^j}{j^\alpha 2^{j-1}}.
\end{eqnarray*}
Since $\frac{1}{(m-1)}\geq \frac{1}{2^{i+1}-1}$ and $\frac{1}{w(m)}=i^\alpha 2^{i-1}$, it follows that
\[
 \frac{1}{(m-1) w(m)}\sum_{n=m}^\infty w(n)\geq \frac{1}{(2^{i+1}-1)}i^\alpha 2^{i-1} 2\sum_{j=i+1}^\infty\frac{1}{j^\alpha}
 =  \frac{i^\alpha 2^i}{(2^{i+1}-1)}\sum_{j=i+1}^\infty \frac{1}{j^\alpha}.
\]
But, $\sum_{j=i+1}^\infty\frac{1}{j^\alpha}\geq \int_{i+1}^\infty \frac{dx}{x^\alpha}=\frac{1}{(\alpha-1)(i+1)^{\alpha-1}}$ and so
\begin{eqnarray*}
& & \frac{1}{(m-1) w(m)}\sum_{n=m}^\infty w(n)\geq \frac{i^\alpha 2^i}{(\alpha-1)(i+1)^{\alpha-1}(2^{i+1}-1} \\
& & =  \frac{i}{(\alpha-1)}\cdot\left(\frac{i}{i+1}\right)^{\alpha-1}\frac{2^i}{2^{i+1}-1}.
\end{eqnarray*}
Since $\lim_{i\to\infty} \left(\frac{i}{i+1}\right)^{\alpha-1}=1$ and $\frac{2^i}{2^{i+1}-1}=\frac{1}{2-2^{-i}}>\frac{1}{2}$, it follows from the previous inequality that
\[
\cU_w=\sup_{m\geq 2}\frac{1}{(m-1) w(m)}\sum_{n=m}^\infty w(n)= \infty.
\]

\textit{Fact 4. $\sC^{(1,w)}$ is not compact.}

This is immediate from Proposition \ref{P.SC-Bound-C}.

\textit{Fact 5. The range of  $I-\sC^{(1,w)}$ is  not closed in $\ell_1(w)$.}

See Facts 1 and 2 and Proposition \ref{L.range_closed}.

\textit{Fact 6. $\ell_1(w)={\rm span}\{\mathbf{1}\}\oplus \ov{(I-\sC^{(1,w)})(\ell_1(w))}$.}

Follows from Facts 1 and 2 and Lemma \ref{L-Mean2}.

\textit{Fact 7. $S_w(1)=(1,\infty)$ and $s_1=1$.}

Fix  $s>0$. From the definition of $w$ we have
\begin{eqnarray*}
\sup_{n\in\N}\frac{1}{n^sw(n)}&=&\max\left\{1, \frac{1}{2^s},\sup_{i\in\N}\left(\max_{n=2^i+1,\ldots,2^{i+1}}\frac{1}{n^sw(n)}\right)\right\}\\
&=&\max\left\{1, \frac{1}{2^s},\sup_{i\in\N}i^\alpha 2^{i-1}\left(\max_{n=2^i+1,\ldots,2^{i+1}}\frac{1}{n^s}\right)\right\}.
\end{eqnarray*}
Since  $\frac{1}{2^{s(i+1)}}\leq \frac{1}{n^s}\leq \frac{1}{2^{si}}$ for all $2^i+1\leq n\leq 2^{i+1}$  and $i\in\N$, it follows that
\[
\max\left\{1, \frac{1}{2^s},\sup_{i\in\N}\frac{1}{2^{s+1}}\frac{i^\alpha}{2^{i(s-1)}}\right\}\leq \sup_{n\in\N}\frac{1}{n^sw(n)}\leq  \max\left\{1, \frac{1}{2^s},\sup_{i\in\N}\frac{1}{2}\cdot\frac{i^\alpha}{2^{i(s-1)}}\right\}.
\]
Accordingly, $\sup_{n\in\N}\frac{1}{n^sw(n)}<\infty$ if and only if $s>1$, i.e.,  $S_w(1)=(1,\infty)$. Hence, $s_1=\inf S_w(1)=1$.

\textit{Fact 8. $\sigma_{pt}(\sC^{(1,w)})=\{1\}$ and $\sigma(\sC^{(1,w)})=\{\lambda\in \C\colon |\lambda-\frac{1}{2}|\leq \frac{1}{2}\}$. }

Since $w\in\ell_1$, we have $1\in \sigma_{pt}(\sC^{(1,w)})$; see Remark \ref{R.1n}. Moreover, $s_1=1$ and so Proposition \ref{P.S_R}(i) implies that $t_0\leq 1$. Then \eqref{e:point} shows that $\sigma_{pt}(\sC^{(1,w)})\su \{1,\frac{1}{2}\}$. Hence, to establish that $\sigma_{pt}(\sC^{(1,w)})=\{1\}$ it suffices to show that $\frac{1}{2}\not\in \sigma_{pt}(\sC^{(1,w)})$, i.e., that $1\not\in R_w$ (see \eqref{e:point}). That this is indeed so follows from the inequalities
\begin{eqnarray*}
\sum_{n=1}^\infty n^1w(n) & = & 2+\sum_{i=1}^\infty\sum_{n=2^i+1}^{2^{i+1}}\frac{n}{i^\alpha 2^{i-1}}\geq 2 +\sum_{i=1}^\infty\sum_{n=2^i+1}^{2^{i+1}}\frac{2^i}{i^\alpha 2^{i-1}}\\
& = & 2+2\sum_{i=1}^\infty\sum_{n=2^i+1}^{2^{i+1}}\frac{1}{i^\alpha }\geq \sum_{i=1}^\infty\frac{2^i}{i^\alpha}=\infty.
\end{eqnarray*}
Hence, the point spectrum $\sigma_{pt}(\sC^{(1,w)})=\{1\}$.

Since $s_1=1$, it follows from Proposition \ref{t:sp-dec}(ii) that
\begin{equation}\label{eq.inclusionO}
\left\{\lambda\in \C\colon \left|\lambda-\frac{1}{2}\right|\leq \frac{1}{2}\right\}\su \sigma(\sC^{(1,w)}).
\end{equation}
For the reverse inclusion,
let $\lambda\in \C$ satisfy $\left|\lambda-\frac{1}{2}\right|> \frac{1}{2}$ and set $\beta:={\rm Re}\left(\frac{1}{\lambda}\right)$.
 Then $\beta<1$, i.e., $(1-\beta)>0$.
Fix  $m\geq 3$ and select $i\in\N$ such that $(2^i+1)< m+1<2^{i+1}$ (note that also  $(2^i+1)\leq m<2^{i+1}$). Then
\[
\sum_{n=m+1}^\infty\frac{w(n)}{n^{1-\beta}}\leq  \sum_{n=2^i+1}^\infty\frac{w(n)}{n^{1-\beta}}=\sum_{j=i}^\infty\sum_{n=2^j+1}^{2^{j+1}}\frac{1}{j^\alpha 2^{j-1}}\cdot\frac{1}{n^{1-\beta}}.
\]
Since $\frac{1}{w(m)}=i^\alpha 2^{i-1}$ with $\frac{1}{m^\beta}\leq (\frac{1}{2^i})^\beta$ and $\frac{1}{n^{1-\beta}}\leq \frac{1}{(2^j)^{1-\beta}}$ for $(2^i+1)\leq n\leq 2^{i+1}$, it follows that
\[
\frac{1}{m^\beta w(m)}\sum_{n=m+1}^\infty\frac{w(n)}{n^{1-\beta}} \leq  \frac{i^\alpha 2^{i-1}}{2^{i\beta}}\sum_{j=i}^\infty \frac{1}{j^\alpha 2^{j-1}}\cdot \frac{1}{(2^j)^{1-\beta}}\cdot 2^j.
\]
But, $\frac{i^\alpha}{j^\alpha}\leq 1$ for all $j\geq i$ and so, for all $m\geq 3$, we have
\[
\frac{1}{m^\beta w(m)}\sum_{n=m+1}^\infty\frac{w(n)}{n^{1-\beta}} \leq 2^{i(1-\beta)}\sum_{j=i}^\infty \left(\frac{1}{2^{1-\beta}}\right)^j=\frac{2^{1-\beta}}{2^{1-\beta}-1}.
\]
On the other hand, recalling that $w\in \ell_1$, we also have  $\frac{1}{ w(1)}\sum_{n=2}^\infty\frac{w(n)}{n^{1-\beta}}\leq \sum_{n=2}^\infty w(n)<\infty$  and $\frac{1}{ 2^\beta w(2)}\sum_{n=3}^\infty\frac{w(n)}{n^{1-\beta}}\leq \frac{1}{2^\beta}\sum_{n=1}^\infty w(n)<\infty$.
Accordingly,
\[
\sup_{m\in\N}\frac{1}{m^\beta w(m)}\sum_{n=m+1}^\infty\frac{w(n)}{n^{1-\beta}}<\infty
\]
and so Theorem \ref{t:sp-dec-N}(ii) implies that $\lambda\in \rho(\sC^{(1,w)})$. Hence, $\{\lambda\in\C\colon |\lambda-\frac{1}{2}|> \frac{1}{2}\}\su  \rho(\sC^{(1,w)})$ which implies that \eqref{eq.inclusionO} is an equality.
\end{example}

It would be interesting to know whether or not $\sC^{(1,w)}$ (equivalently, $\sC^{(1,v)}$; see Fact 1) is power bounded.

Concerning the dynamics of a continuous linear operator $T$ defined on a separable Banach space $X$,  recall that $T$ is \textit{hypercyclic} if there exists $x\in X$ such that the orbit $\{T^nx\colon n\in\N_0\}$ is dense in $X$. If, for some $x\in X$, the projective orbit $\{\lambda T^nx\colon \lambda\in\C,\ n\in\N_0\}$ is dense in $X$, then $T$ is called \textit{supercyclic}. Clearly, hypercyclicity  implies supercyclicity.

\begin{prop} Let   $w$ be   a bounded, strictly positive sequence such that $\sC^{(1,w)}\in \cL(\ell_1(w))$. Then  $\sC^{(1,w)}$  is not supercyclic and so, not hypercyclic.
\end{prop}

\begin{proof}
By Step 2 in the proof of Theorem \ref{t:sp-dec-N}
the \textit{infinite} set  $\Sigma\su \sigma_{pt}((\sC^{(1,w)})')$.  By Theorem 3.2 of  \cite{AB},  $\sC^{(1,w)}$ cannot be \textit{supercyclic}.
\end{proof}

\bigskip

\textbf{Acknowledgements.} The research of the first two authors was partially supported by
the projects MTM2013-43540-P and MTM2016-76647-P.

The second author acknowledges the support of the ``International Visiting Professor Program 2016'' from the Ministry of Education, Science and Art, Bavaria (Germany).

\bigskip
\bibliographystyle{plain}

\end{document}